\documentclass[11pt]{amsart}

\usepackage[T1]{fontenc}
\usepackage[latin1]{inputenc}
\usepackage[english]{babel}
\usepackage{amsmath,amssymb,amsthm}
\usepackage{enumitem}
\usepackage{tikz}
\usepackage{hyperref}

\renewcommand{\MR}[1]{}

\newtheorem{theorem}{Theorem}[section]
\newtheorem{lemma}[theorem]{Lemma}
\newtheorem{proposition}[theorem]{Proposition}
\newtheorem{corollary}[theorem]{Corollary}

\theoremstyle{definition}
\newtheorem{definition}[theorem]{Definition}

\newtheorem{question}[theorem]{Question}
\newtheorem{remark}[theorem]{Remark}
\newtheorem{example}[theorem]{Example}
\newtheorem{conjecture}[theorem]{Conjecture}
\newtheorem*{remark*}{Remark}

\title{An invitation to the Drury--Arveson space}
\author[M. Hartz]{Michael Hartz}
\address{Fachrichtung Mathematik, Universit\"at des Saarlandes, 66123 Saarbr\"ucken, Germany}
\email{hartz@math.uni-sb.de}
\thanks{The author was partially supported by a GIF grant and by the Emmy Noether Program
of the German Research Foundation (DFG Grant 466012782).}
\dedicatory{Dedicated to the memory of J\"org Eschmeier}
\date{\today}


\newcommand{\mcc}{M\textsuperscript{c}Carthy}
\newcommand{\diag}{\operatorname{diag}}
\newcommand{\Mult}{\operatorname{Mult}}
\newcommand{\ran}{\operatorname{ran}}
\newcommand{\Han}{\operatorname{Han}}
\newcommand{\Alg}{\operatorname{Alg}}
\newcommand{\Lat}{\operatorname{Lat}}
\renewcommand{\Re}{\operatorname{Re}}
\newcommand{\spa}{\operatorname{span}}

\begin{document}

\begin{abstract}
  This is an extended version of a three part mini course on the Drury--Arveson space given as part of the Focus Program on Analytic Function Spaces and their Applications, hosted by the Fields Institute and held remotely.
  The Drury--Arveson space, also known as symmetric Fock space, is a natural generalization of the classical Hardy space on the unit disc to the unit ball
  in higher dimensions. It plays a universal role both in operator theory and in function theory.
  These notes give an introduction to the Drury--Arveson space.
  They are not intended to give a comprehensive overview of the entire subject, but rather aim to explain some of the contexts in which the Drury--Arveson space makes an appearance and to showcase the different approaches to this space.
\end{abstract}

\maketitle

\tableofcontents

\section{Introduction}

    Let $\mathbb{B}_d = \{z \in \mathbb{C}^d: \|z\|_2 < 1 \}$ denote the Euclidean unit ball in $\mathbb{C}^d$.
    The Drury--Arveson space $H^2_d$ is a space of holomorphic functions on $\mathbb{B}_d$
    that generalizes the classical Hardy space $H^2$ on the unit disc to several variables.
    Given the remarkable success of the theory of $H^2$,
    it is natural to study generalizations to several variables.
    In fact, there are deeper reasons that lead to the Drury--Arveson space.
 In these notes, we focus on two particular prominent ones:
      \begin{enumerate}
        \item Operator theory: $H^2_d$ hosts an analogue of the unilateral shift that is universal for commuting row contractions.
        \item Function theory: $H^2_d$ is a universal space among certain Hilbert spaces of functions (complete Pick spaces).
      \end{enumerate}
      Thus, $H^2_d$ is a universal object in two separate areas.

      Just as in the case of $H^2$, the theory of $H^2_d$
      turns out to be very rich because of several different viewpoints on this object.
      We will see that $H^2_d$ can be understood as
      \begin{enumerate}[label=\normalfont{(\alph*)}]
        \item a space of power series,
        \item a reproducing kernel Hilbert space with a particularly simple kernel,
        \item a particular Besov--Sobolev space, and
        \item symmetric Fock space.
      \end{enumerate}
      Different points of view reveal different aspects of $H^2_d$, and often
      a given problem becomes more tractable by choosing the appropriate point of view.

      The target audience of these notes includes graduate students and newcomers to the subject.
      In particular, these notes are intended to be an introduction to the Drury--Arveson space
      and do not aim to give a comprehensive overview.
      For further information and somewhat different points of view, the reader is referred to the survey articles
   by Shalit \cite{Shalit13} and by Fang and Xia \cite{FX19}. Material regarding the operator theoretic and operator algebraic aspects of the Drury--Arveson space can be found in Arveson's original article \cite{Arveson98} and in the book by Chen and Guo \cite{CG03}.
   For the appearance of the Drury--Arveson space in the context of complete Pick spaces, the reader is referred to the book by Agler and \mcc\ \cite{AM02}.
   For a harmonic analysis point of view, see the book by Arcozzi, Rochberg, Sawyer and Wick \cite{ARS+19}.

      In an attempt to make these notes more user friendly,
      a few deliberate choices were made.
      Firstly, the presentation is often not the most economical one;
      rather, the goal is to give an idea of how one might come across certain ideas
      and how various results are connected.
      Secondly, instead of discussing results in their most general form,
      we often restrict attention to instructive special cases.
      In the same vein, simplifying assumptions are made whenever convenient.

  We end this introduction by mentioning a few early appearances of the Drury--Arveson space $H^2_d$.
  In comparison to some of the other function spaces covered during the Focus Program, $H^2_d$ was systematically studied only much later.
  The name ``Drury--Arveson space'' comes from a 1978 paper of Drury \cite{Drury78}
  and a 1998 paper of Arveson \cite{Arveson98}, both of which were motivated by operator theoretic questions.
  Arveson's paper brought $H^2_d$ to prominence, explicitly realized $H^2_d$ as a natural
  reproducing kernel Hilbert space, and connected it to symmetric Fock space.
  Early appearances of $H^2_d$ can also be found in papers of Lubin \cite{Lubin76,Lubin77}
  and work of M\"uller and Vasilescu \cite{MV93}. Around the turn of the millennium,
  the Drury--Arveson space and related non-commutative objects were also studied by Davidson and Pitts \cite{DP98}
  and by Arias and Popescu \cite{AP00}, and the universal role of $H^2_d$ in the context of interpolation problems
  was established by Agler and \mcc\ \cite{AM00,AM00a}.
  After all these developments, the subject gained a lot of momentum, and has been very active to this day.

  \section{Several definitions of \texorpdfstring{$H^2_d$}{Drury--Arveson space}}

  \subsection{Hardy space preliminaries}
  \label{ss:Hardy}
  As the Drury--Arveson space is a generalization of the classical Hardy space $H^2$,
  let us recall a few basic facts about $H^2$.
  At the Focus Program, a mini course on the Hardy space was given by Javad Mashreghi.
  For more background, see also the books \cite{Duren70,Garnett07,Hoffman62,Koosis98,Nikolski19}.

  Let $\mathbb{D} = \{z \in \mathbb{C}: |z| < 1 \}$ be the open unit disc and let
  $\mathcal{O}(\mathbb{D})$ be the algebra of all holomorphic functions on $\mathbb{D}$.
  A common definition of the Hardy space is
  \begin{equation}
    \label{eqn:defn}
    H^2 = \Big\{f \in \mathcal{O}(\mathbb{D}): \|f\|^2_{H^2} = \sup_{0 \le r < 1} \int_{0}^{2 \pi} |f(r e^{ it})|^2 \frac{dt}{2 \pi} < \infty \Big\}.
  \end{equation}
  If $f(z) = \sum_{n=0}^\infty a_n z^n$ is the Taylor expansion
  of a holomorphic function $f$ on $\mathbb{D}$, then
  \begin{equation*}
    \|f\|^2_{H^2} = \sum_{n=0}^\infty |a_n|^2.
  \end{equation*}

  Recall that a \emph{reproducing kernel Hilbert space (RKHS)} is a Hilbert space $\mathcal{H}$ of functions
  on a set $X$ with the property that evaluation at each $x \in X$ is a bounded linear functional
  on $\mathcal{H}$. By the Riesz representation theorem, there exists for each $x \in X$
  a function $k_x \in \mathcal{H}$ such that
  \begin{equation*}
    \langle f, k_x \rangle  = f(x) \quad \text{ for all }  f \in \mathcal{H}.
  \end{equation*}
  The function $k:X \times X \to \mathbb{C}, k(x,y) = k_y(x),$ is called the \emph{reproducing kernel} or simply
  \emph{kernel} of $\mathcal{H}$.
  For background material on RKHS, see the classic paper \cite{Aronszajn50} and the books \cite{AM02,PR16}.

  The Hardy space $H^2$ is an RKHS of functions on $\mathbb{D}$; the reproducing kernel is the Szeg\H{o} kernel
  \begin{equation*}
    K(z,w) = \frac{1}{1 - z \overline{w}} \quad (z,w \in \mathbb{D}).
  \end{equation*}
  
  \subsection{Discovering the Drury--Arveson space}
  \label{ss:discover}
  There is a classical generalization of the Hardy space $H^2$ to the unit ball, given by
  mimicking \eqref{eqn:defn} above.
  Let $\sigma$ be the normalized surface measure on $\partial \mathbb{B}_d$. The \emph{Hardy space on $\mathbb{B}_d$} is defined to be
    \begin{equation*}
      H^2(\mathbb{B}_d) =
      \Big\{ f \in \mathcal{O}(\mathbb{B}_d): \|f\|^2_{H^2(\mathbb{B}_d)}
        = \sup_{0 \le r < 1} \int_{\partial \mathbb{B}_d} |f( r z)|^2 d \sigma(z)
      < \infty \Big\},
    \end{equation*}
    see for instance \cite[Section 5.6]{Rudin08} or \cite[Chapter 4]{Zhu05} for background on $H^2(\mathbb{B}_d)$.

  It turns out that for the purposes of multivariable operator theory, $H^2(\mathbb{B}_d)$ is not
  as close an analogue of $H^2$ as one might like.
  We will later see this more concretely; see Subsection \ref{ss:dilation} below. For now, recall
  that one of the key aspects of $H^2$ is that $\|z^n\|_{H^2} = 1$ for all $n \in \mathbb{N}$.
  However, if $d \ge 2$, then
  \begin{equation}
    \label{eqn:Hardy_norm}
    \|z_1^n\|_{H^2(\mathbb{B}_d)} = \int_{\partial \mathbb{B}_d} |z_1|^{2 n} d \sigma(z) \xrightarrow{n \to \infty} 0,
  \end{equation}
  which for instance follows from the dominated convergence theorem.
  (The integral can be computed explicitly; see for instance \cite[Proposition 1.4.9]{Rudin08}.)
  In particular, we do not recover $H^2$ from $H^2(\mathbb{B}_d)$ when considering functions of the single variable $z_1$.

  Taking a cue from this simple observation, we ask:

  \begin{question}
    \label{goal}
    Does there exist a reproducing kernel Hilbert space $\mathcal{H}$ of holomorphic functions on $\mathbb{B}_d$ satisfying
    all of the following properties:
    \begin{enumerate}[label=\normalfont{(\alph*)}]
      \item if $f(z) = g(z_1)$ for some $g \in H^2$, then $f \in \mathcal{H}$ and $\|f\| = \|g\|_{H^2}$;
      \item if $U$ is a $d \times d$ unitary and $f \in \mathcal{H}$, then $f \circ U \in \mathcal{H}$ and $\|f \circ U\| = \|f\|$;
      \item the polynomials are dense in $\mathcal{H}$?
    \end{enumerate}
  \end{question}

  We will see that there is a unique space $\mathcal{H}$ satisfying the conditions
  in Question \ref{goal}, and this space is the Drury--Arveson space.
  Note that $H^2(\mathbb{B}_d)$ satisfies the last two conditions, but by Equation \eqref{eqn:Hardy_norm},
  it fails Condition (a).

  \begin{remark}
    Clearly, by unitary invariance (Condition (b)), $z_1$ can be replaced by any $z_j$,
    or more generally by $\langle z,\zeta \rangle$ for $\zeta \in \partial \mathbb{B}_d$
    in Condition (a).
  \end{remark}

  Let us see what a space satisfying the Conditions in Question \ref{goal} must look like.
  To this end, we recall multi-index notation.
  If $\alpha = (\alpha_1,\ldots,\alpha_d) \in \mathbb{N}^d$, we write
  \begin{equation*}
    z^\alpha = z_1^{\alpha_1} \ldots z_d^{\alpha_d}, \quad \alpha! = \alpha_1! \cdot \ldots \cdot \alpha_d!, \quad |\alpha| = \alpha_1 + \ldots + \alpha_d.
  \end{equation*}

  \begin{lemma}
    If $\mathcal{H}$ satisfies the conditions of Question \ref{goal}, then the monomials $\{z^\alpha: \alpha \in \mathbb{N}^d\}$ form an orthogonal basis of $\mathcal{H}$.
  \end{lemma}

  \begin{proof}
  Let $\alpha \neq \beta$, say $\alpha_1 \neq \beta_1$, and let
  $U_t = \diag(e^{ it} , 1 ,\ldots ,1)$ for $t \in \mathbb{R}$.
  Condition (b) shows that $f \mapsto f \circ U_t$ defines an isometry on $\mathcal{H}$.
  Hence
  \begin{equation*}
    \langle z^\alpha, z^\beta \rangle = \langle z^\alpha \circ U_t, z^\beta \circ U_t \rangle = e^{i (\alpha_1 - \beta_1) t} \langle z^\alpha, z^\beta \rangle .
  \end{equation*}
  Integrating in $t$ gives $\langle z^\alpha, z^\beta \rangle = 0$,
  so the monomials form an orthogonal system.
  Since the polynomials are dense by Condition (c), they are in fact an orthogonal basis.
  \end{proof}

  In order to specify $\mathcal{H}$, it remains to determine the norm of the monomials.
  Note that Condition (a) forces $\|z_1^n\| = 1$ for all $n \in \mathbb{N}$.
  In combination with Condition (b), this implies that
  \begin{equation*}
    \| \langle z, \zeta \rangle^n\| = 1
  \end{equation*}
  for all $\zeta \in \partial \mathbb{B}_d$, from which one can compute the norm of each monomial.
  For instance, to compute the norm of $z_1 z_2$, we set $\zeta = (1/\sqrt{2},1/\sqrt{2},0,\ldots)$
  and $n=2$ above to find that
  \begin{equation*}
    1 = \| \tfrac{1}{2} (z_1 + z_2)^2 \|^2 = \tfrac{1}{4} \|z_1^2\|^2 + \|z_1 z_2\|^2 + \tfrac{1}{4} \|z_2^2\|^2,
  \end{equation*}
  so using that $\|z_j^2\|^2 = 1$, it follows that $\|z_1 z_2\|^2 = \tfrac{1}{2}$.

  In the general case, we will argue slightly differently, using the Szeg\H{o} kernel of $H^2$.
  Let $\lambda \in \mathbb{D}$ and $f = \sum_{\alpha \in \mathbb{N}^d} a_\alpha z^\alpha \in \mathcal{H}$.
  Condition (a)
  implies that 
  \begin{equation*}
    \frac{1}{1 - \overline{\lambda} z_1} = \sum_{n=0}^\infty \overline{\lambda}^n z_1^n \in \mathcal{H},
  \end{equation*}
  that the sum converges in $\mathcal{H}$, and that
  $\|z_1^n\|=1$ for all $n \in \mathbb{N}$. Hence,
  letting $e_1 = (1,0,\ldots, 0) \in \mathbb{C}^d$, we have
\begin{equation*}
  \Big\langle f, \frac{1}{1 - \overline{\lambda} z_1} \Big\rangle
  = \sum_{n=0}^\infty a_{(n,0,\ldots,0)} \lambda^n \|z_1^n\|^2
  = f(\lambda e_1).
\end{equation*}
If $w \in \mathbb{B}_d$, let $U$ be a $d \times d$ unitary and $\lambda \in \mathbb{D}$ with $U^* w = \lambda e_1$.
Then by unitary invariance,
\begin{equation}
  \label{eqn:kernel}
f(w) = \Big \langle f \circ U, \frac{1}{1 - \overline{\lambda} z_1} \Big\rangle 
  = \Big \langle f, \frac{1}{1 - \overline{\lambda} z_1} \circ U^* \Big\rangle 
  = \Big \langle f, \frac{1}{1 - \langle z,w \rangle_{\mathbb{C}^d} } \Big\rangle.
\end{equation}
Now the multinomial expansion shows that
\begin{equation*}
  \frac{1}{1 - \langle z,w \rangle_{\mathbb{C}^d} }
  = \sum_{n=0}^\infty \langle z,w \rangle^n_{\mathbb{C}^d} = \sum_{n=0}^\infty \sum_{|\alpha| = n} \frac{|\alpha|!}{\alpha!}
  \overline{z}^\alpha \overline{w}^\alpha,
\end{equation*}
and the sum converges in $\mathcal{H}$ by unitary invariance and the earlier case $w \in \mathbb{D} e_1$.
Thus, applying \eqref{eqn:kernel} to $f= z^\alpha$, we see that
\begin{align}
  \label{eqn:kernel_2}
  w^\alpha &= \Big \langle z^\alpha, \frac{1}{1 - \langle z,w \rangle_{\mathbb{C}^d} } \Big\rangle
  = w^\alpha \|z^\alpha\|^2 \frac{|\alpha|!}{\alpha!}
 \end{align}
 for all $w \in \mathbb{B}_d$.
  Therefore, we arrive at the following conclusion:
  \begin{proposition}
    \label{prop:power_series_nec}
    If $\mathcal{H}$ is a space as in Question \ref{goal}, then the monomials $z^\alpha$
    form an orthogonal basis of $\mathcal{H}$, and
    \begin{equation*}
      \|z^\alpha\|^2 = \frac{\alpha!}{|\alpha|!}.
    \end{equation*}
  \end{proposition}

  Equation \eqref{eqn:kernel} shows another important feature of $\mathcal{H}$ (if it exists), namely
  that the reproducing kernel of $\mathcal{H}$ is given by
  \begin{equation*}
    K(z,w) = \frac{1}{1 - \langle z,w \rangle_{\mathbb{C}^d}}
  \end{equation*}

  \subsection{Power series description of the Drury--Arveson space}
  We now turn the argument in the last section around and take Proposition \ref{prop:power_series_nec}
  as our definition.

  \begin{definition}
    The Drury--Arveson space is defined to be
    \begin{equation*}
      H^2_d = \Big\{ f = \sum_{\alpha \in \mathbb{N}^d} a_\alpha z^\alpha: \|f\|^2 = \sum_{\alpha \in \mathbb{N}^d} |a_\alpha|^2 \frac{\alpha!}{|\alpha|!} < \infty \Big\}.
    \end{equation*}
  \end{definition}

  The weights $\frac{\alpha!}{|\alpha|!}$ appearing in the definition may look strange at first.
  However, the argument in the last section shows that their appearance may be explained by unitary
  invariance of the space and hence by the symmetry of the unit ball.

  Strictly speaking, the above definition exhibits $H^2_d$ as a space of formal power series.
  But just as in the case of $H^2_d$, the power series converge and define bona fide holomorphic
  functions on the open ball.

  \begin{proposition}
    \label{prop:DA_point_estimate}
    If $f = \sum_{\alpha \in \mathbb{N}^d} a_\alpha z^\alpha \in H^2_d$,
    then the series converges absolutely and uniformly on every compact subset of $\mathbb{B}_d$.
    Moreover,
    \begin{equation}
      \label{eqn:pointwise}
      |f(w)|^2 \le \frac{1}{1 - \|w\|^2} \|f\|^2 \quad \text{ for all } w \in \mathbb{B}_d.
    \end{equation}
  \end{proposition}

  \begin{proof}
    Let $w \in \mathbb{B}_d$.
    By the Cauchy--Schwarz inequality and the binomial theorem,
    \begin{align*}
      \sum_{\alpha \in \mathbb{N}^d} |a_\alpha w^\alpha|
      &\le
      \Big( \sum_\alpha |w^\alpha|^2 \frac{|\alpha|!}{\alpha!} \Big)^{1/2}
      \Big( \sum_\alpha |a_\alpha|^2 \frac{\alpha!}{|\alpha|!} \Big)^{1/2} \\
      &= \frac{1}{(1 - \|w\|^2)^{1/2}} \|f\|.
    \end{align*}
    This shows absolute convergence of the series and the pointwise estimate \eqref{eqn:pointwise}.
    Applying the pointwise estimate to the tail of the series,
    we see that the convergence
    is uniform on $r \overline{\mathbb{B}_d}$ for each $r < 1$.
  \end{proof}

  It also follows that the expansion of $f$ in the orthogonal basis of monomials
  agrees with the Taylor series of $f$ at the origin.

  \subsection{RKHS description of the Drury--Arveson space}

  We already came across the reproducing kernel of the Drury--Arveson space
  in our discussion in Subsection \ref{ss:discover}.
  Let us check that with our formal definition of $H^2_d$, we really obtain the same
  reproducing kernel.

  \begin{proposition}
    \label{prop:kernel}
    The Drury--Arveson space is a reproducing kernel Hilbert space
    of functions on $\mathbb{B}_d$ with reproducing kernel
    \begin{equation*}
      K(z,w) = \frac{1}{1 - \langle z,w \rangle } \quad (z,w \in \mathbb{B}_d).
    \end{equation*}
  \end{proposition}

  \begin{proof}
    It is elementary to check that for each $w \in \mathbb{B}_d$, the function 
    \begin{equation*}
      K_w(z) := \frac{1}{1 - \langle z,w \rangle } = \sum_\alpha \frac{|\alpha|!}{\alpha!} z^\alpha \overline{w^\alpha}
    \end{equation*}
    belongs
    to $H^2_d$. Moreover, essentially reading the computation \eqref{eqn:kernel_2} backwards,
    we find that if $f = \sum_{\alpha} a_\alpha z^\alpha \in H^2_d$, then
    \begin{equation*}
      \langle f, K_w \rangle = \sum_\alpha a_\alpha \Big\langle z^\alpha, \frac{|\alpha|!}{\alpha!} z^\alpha \overline{w}^\alpha \Big\rangle  = \sum_\alpha a_\alpha w^\alpha = f(w). \qedhere
    \end{equation*}
  \end{proof}

  \begin{remark}
    Proposition \ref{prop:kernel} could also serve as the definition of $H^2_d$.
    Indeed,
    \begin{equation*}
      K(z,w) = \frac{1}{1 - \langle z,w \rangle} = \sum_{\alpha} \frac{|\alpha|!}{\alpha!} z^\alpha \overline{w^\alpha}
    \end{equation*}
    is positive semi-definite as a function of $(z,w)$ on $\mathbb{B}_d \times \mathbb{B}_d$.
    Thus, Moore's theorem (see \cite[Theorem 2.14]{PR16} or \cite[Theorem 2.23]{AM02})
    implies that there is a unique reproducing kernel Hilbert space of functions on $\mathbb{B}_d$
    with reproducing kernel $K$.
  \end{remark}

  The RKHS description of $H^2_d$ gives a strong indication that it is a natural space:
  the reproducing kernel in Proposition \ref{prop:kernel} is perhaps the simplest
  generalization of the Szeg\H{o} kernel to the unit ball that one can think of.

  We can now also see that $H^2_d$ indeed answers Question \ref{goal}.

  \begin{proposition}
    \label{prop:unit_invariance}
    The Drury--Arveson space $H^2_d$ satisfies all conditions in Questions \ref{goal}.
    In particular, $H^2_d$ is unitarily invariant in the sense that
    for any $d \times d$ unitary matrix $U$, the composition operator
    \begin{equation*}
      C_U: H^2_d \to H^2_d, \quad f \mapsto f \circ U,
    \end{equation*}
    is unitary.
  \end{proposition}

  \begin{proof}
    Conditions (a) and (c) hold by definition, and Proposition \ref{prop:DA_point_estimate}
    shows that $H^2_d$ is indeed a reproducing kernel Hilbert space of holomorphic functions on $\mathbb{B}_d$.

    To see unitary invariance, i.e.\ Condition (b), observe that the reproducing kernel $K$ of $H^2_d$
    is unitarily invariant in the sense that for all $d \times d$ unitaries $U$,
    we have
    \begin{equation}
      \label{eqn:kernel_unitary_invariance}
      K(U z, U w) = K(z,w) \quad \text{ for all } z,w \in \mathbb{B}_d
    \end{equation}
    A general RKHS argument then implies that $C_U$ is unitary.
    Indeed, writing $K_w(z) = K(z,w)$, Equation \eqref{eqn:kernel_unitary_invariance} implies that
    \begin{equation*}
      \langle K_{U w}, K_{U v} \rangle = \langle K_w, K_v \rangle 
      \quad \text{ for all } v,w \in \mathbb{B}_d,
    \end{equation*}
    so using that the linear span of the kernel functions is dense in $H^2_d$,
    we obtain a unique isometry $V: H^2_d \to H^2_d$ with
    \begin{equation*}
      V K_w = K_{U w} \quad \text{ for all } w \in \mathbb{B}_d.
    \end{equation*}
    Since $U$ is surjective, $V$ is in fact unitary, and
    \begin{equation*}
      (V^* f)(w) = \langle V^* f, K_w \rangle = \langle f, K_{U w} \rangle = f(U w),
    \end{equation*}
    so $V^* = C_U$.
  \end{proof}

  \subsection{Function theory description of the Drury--Arveson space}
  \label{ss:ft_dA}

  In classical Hardy space theory, the function theoretic point of view plays a central role.
  Therefore, it is desirable to have a function theoretic description
  of the Drury--Arveson space as well.

  To motivate the function theoretic description, let us first recall another
  classical function space on the ball.
  \begin{definition}
    Let $V$ be the normalized volume measure on $\mathbb{B}_d$.
    The Bergman space on $\mathbb{B}_d$ is
    \begin{equation*}
      L^2_a(\mathbb{B}_d) = \Big\{ f \in \mathcal{O}(\mathbb{B}_d): \|f\|^2_{L^2_a(\mathbb{B}_d)} = \int_{\mathbb{B}_d} |f|^2 d V < \infty \Big\}.
    \end{equation*}
  \end{definition}

  Background material on $L^2_a(\mathbb{B}_d)$ can for instance be found in \cite[Chapter 2]{Zhu05}.

  Let us now try to determine how the Drury--Arveson space compares to the Bergman space,
  and let us first assume that $d=2$.
  (One can also compare $H^2_d$ to the Hardy space $H^2(\mathbb{B}_d)$, but in the first
  non-trivial case $d=2$, the comparison with the Bergman space turns out to be cleaner.)

  A basic integral formula on the ball (Propositions 1.4.8 and 1.4.9 in \cite{Rudin08}) shows that the monomials $z^\alpha$ also form
  an orthogonal basis of $L^2_a(\mathbb{B}_2)$, but now with norm
  \begin{equation}
    \label{eqn:norm}
    \|z^\alpha\|^2_{L^2_a(\mathbb{B}_2)} = \frac{\alpha!}{|\alpha|!} \frac{2}{(|\alpha|+1)(|\alpha|+2)}
    \approx \frac{\alpha!}{|\alpha|!} \frac{1}{|\alpha|^2}
    = \|z^\alpha\|^2_{H^2_2} \frac{1}{|\alpha|^2}
  \end{equation}
  if $\alpha \in \mathbb{N}^d \setminus \{0\}$.
  In other words, the $H^2_2$-norm of $z^\alpha$ exceeds its $L^2_a$-norm by roughly
  a factor of $|\alpha|$.
  Crucially, this additional factor has a concrete function theoretic interpretation.

  If $f \in \mathcal{O}(\mathbb{B}_d)$, we define the \emph{radial derivative} of $f$ to be
  \begin{equation*}
    R f = \sum_{j=1}^d z_j \frac{\partial f}{\partial z_j}.
  \end{equation*}
  Observe that
  \begin{equation*}
    R z^\alpha = |\alpha| z^\alpha,
  \end{equation*}
  so Equation \eqref{eqn:norm} implies that
  \begin{equation}
    \label{eqn:H22}
    H^2_2 = \Big\{ f \in \mathcal{O}(\mathbb{B}_2): R f \in L^2_a(\mathbb{B}_2) \Big\}
  \end{equation}
  and that
  \begin{equation}
    \|f\|^2_{H^2_2} \approx |f(0)|^2 + \int_{\mathbb{B}_2} |R f|^2 \, d V.
  \end{equation}

  This idea can be generalized to higher dimensions, where one needs to take radial derivatives
  of higher order.
  The resulting statement is the following.

  \begin{theorem}
    \label{thm:DA_FT}
    Let $m \in \mathbb{N}$ with $2 m - d > -1$.
    Then
    \begin{equation*}
      H^2_d = \Big\{f \in \mathcal{O}(\mathbb{B}_d): \int_{\mathbb{B}_d} |R^m f(z)|^2 (1 - |z|^2)^{2 m - d} d V(z) < \infty \Big\}
    \end{equation*}
    and
    \begin{equation*}
      \|f\|^2_{H^2_d} \approx |f(0)|^2 + \int_{\mathbb{B}_d} |R^m f(z)|^2 (1 - |z|^2)^{2 m - d} d V(z).
    \end{equation*}
  \end{theorem}
  Note that if $d$ is even, then we may choose $m = d/2$ above, so in this case, 
  \begin{equation*}
    \|f\|^2_{H^2_d} \approx |f(0)|^2 + \int_{\mathbb{B}_d} |R^{d/2} f|^2 \,d V.
  \end{equation*}

  The idea behind the proof of Theorem \ref{thm:DA_FT} is similar to one in the special case $d=2$ and $m=1$ outlined above.
  The details can be found, for instance, in \cite{ZZ08}, see in particular Corollary 11 and Theorem 41 there.

  \begin{remark}
  The formula for the norm in Theorem \ref{thm:DA_FT} suffers from the defect that it
  is not exact.
  Explicitly, the right-hand side in Theorem \ref{thm:DA_FT} only defines an equivalent norm on $H^2_d$.
  Moreover, the implied constants depend on the dimension $d$,
  which poses an issue if one tries to use Theorem \ref{thm:DA_FT} in order to analyze functions in $H^2_\infty$;
  see Definition \ref{def:DA_infinite} below.
  However, there is an exact function theoretic formula for $\|f\|_{H^2_d}$ using more complicated differential operators by Arcozzi, Monguzzi, Peloso and Salvatori; see \cite[Theorem 6.1]{AMP+19} for the precise statement.
  Furthermore, Greene and Richter showed that one can compute the norm of a function $f \in H^2_d$ by integrating the norms of the slice
  functions $f_\zeta: \mathbb{D} \to \mathbb{C}, f_\zeta(z) = f(z \zeta)$, in suitable spaces on the disc over $\zeta \in \partial \mathbb{B}_d$; see \cite[Theorem 4.2]{GR06} for the precise statement.
  \end{remark}

  It is worth comparing Theorem \ref{thm:DA_FT}, and even the special case $d=2$ in \eqref{eqn:H22},
  with the definition of the Dirichlet space on the unit disc, namely
  \begin{equation*}
    \mathcal{D} = \{f \in \mathcal{O}(\mathbb{D}): R f \in L^2_a(\mathbb{D}) \}.
  \end{equation*}
  This leads to a useful heuristic: If $d \ge 2$, then from a function theoretic
  point of view, the Drury--Arveson space behaves more like the Dirichlet space than like the Hardy space.
  We will see concrete manifestations of this principle later.
  At the Focus Program, a mini course on the Dirichlet space was given by Thomas Ransford.
  Background material on the Dirichlet space can be found in the books \cite{ARS+19,EKM+14}.

\subsection{The Drury--Arveson space as a member of a scale of spaces}
\label{ss:scale}

The Drury--Arveson space can be regarded as one point in an important scale of spaces, which we denote by $\mathcal{H}_a$ for $a \in \mathbb{R}$. These spaces are sometimes called Besov--Sobolev or Hardy--Sobolev spaces,
see for instance \cite{ZZ08} for background.

Pictorially, a few important members of the scale can be identified as follows:
  \begin{center}
  \begin{tikzpicture}[scale=1.2]
    \node at (0.5,0.4) {$a$};
    \node at (.5,0) [below]{$\mathcal H_a$};
    \draw (1,0) -- (10,0);
    \draw (2,0.1) -- (2,-0.1);
   \node at (2,0) [below]{Dirichlet};
    \node at (2,.4) {$0$};
    \draw (4,0.1) -- (4,-0.1);
   \node at (4,0) [below]{Drury--Arveson};
    \node at (4,.4) {$1$};
    \draw (7,0.1) -- (7,-0.1);
   \node at (7,0) [below]{Hardy};
    \node at (7,.4) {$d$};
    \draw (9,0.1) -- (9,-0.1);
   \node at (9,0) [below]{Bergman};
    \node at (9,.4) {$d+1$};
  \end{tikzpicture}
  \end{center}

  Just as for the Drury--Arveson space, there are several descriptions of $\mathcal{H}_a$.
  Using power series, one can define
  \begin{equation*}
    \mathcal{H}_a = \Big\{ f(z) = \sum_{\alpha} a_\alpha z^\alpha \in \mathcal{O}(\mathbb{B}_d): \sum_{\alpha} |a_\alpha|^2 \frac{\alpha!}{|\alpha|!} (|\alpha| + 1)^{1-a} < \infty \Big\}.
  \end{equation*}
  If $a > 0$, then there is an equivalent norm on $\mathcal{H}_a$ so that the reproducing kernel of $\mathcal{H}_a$
  is given by
  \begin{equation*}
    \frac{1}{(1 - \langle z,w \rangle)^a }.
  \end{equation*}
  A function theoretic definition can be obtained as follows.
  If $m \in \mathbb{N}$ with $2 m + a -d > 0$, then
  \begin{equation*}
    \mathcal{H}_a = \Big\{f \in \mathcal{O}(\mathbb{B}_d): \int_{\mathbb{B}_d} |R^m f (z)|^2 (1 - |z|^2)^{ 2m + a - d - 1} d V(z) < \infty \Big\}.
  \end{equation*}
  The proof of the equivalence of these descriptions is similar to the one in the case
  of the Drury--Arveson space. For details, see \cite{ZZ08}, especially Corollary 11, Theorem 40 and Theorem 41 there.

  In the literature, one can find many ways of indexing the spaces $\mathcal{H}_a$.
  The choice of indexing taken here is motivated by the formula for the reproducing kernel.

  \subsection{The non-commutative approach}
  \label{ss:nc}

  The Drury--Arveson space can also be naturally obtained from a space of non-commutative holomorphic
  functions. Remarkably, this point of view has proved to be very useful, even when one is only interested
  in classical (commutative) holomorphic functions.
  Non-commutative function theory and spaces of non-commutative functions were the topic of the mini course given by Michael Jury at the Focus Program.

  Let $x = (x_1,\ldots,x_d)$ be non-commuting indeterminates and
  let $F_d^+$ be the set of all finite words in $\{1,\ldots,d\}$,
  including the empty word.
  If $w = w_1 \ldots w_r \in F_d^+$, set
  $x^w = x_{w_1} x_{w_2} \ldots x_{w_r}$.
  We also write $|w| = r$ for the length of the word $w$.
  \begin{definition}
    The noncommutative Hardy space is
    \begin{equation*}
      H^2_{nc} = \Big \{ F = \sum_{w \in F_d^+} a_w x^w : \|F\|^2 = \sum_{w \in F_d^+} |a_w|^2 < \infty \Big\}.
    \end{equation*}
  \end{definition}

  The elements of $H^2_{nc}$ are formal non-commutative power series.
  If $d=1$, then the power series converge on the unit disc.
  For $d \ge 1$, it makes sense to evaluate the power series not only on scalar tuples,
  but on certain matrix tuples.  If $n \in \mathbb{N}, n\ge 1$
  and $X =
  \begin{bmatrix}
    X_1 & \cdots & X_d
  \end{bmatrix}$ 
  is a $d$ tuple of $n \times n$ matrices, then we can evaluate a monomial $x^w$ at $X$
  by defining $X^w = X_{w_1} \ldots X_{w_r}$. We view $X$ as a linear map
  \begin{equation*}
    X: (\mathbb{C}^n)^d \mapsto \mathbb{C}^n \quad
    \begin{bmatrix}
      v_1 \\ \vdots \\ v_n
    \end{bmatrix} \mapsto \sum_{i=1}^n X_i v_i
  \end{equation*}
  and let $\|X\|$ denote the operator norm of this linear map.
  Let
  \begin{equation*}
    \mathbb{B}_d^{n \times n} = \{ X \in M_n(\mathbb{C})^d: \|X\| < 1\}
  \end{equation*}
  and $\mathbb{B}_d^{nc} = \bigcup_{n=1}^\infty \mathbb{B}_d^{n \times n}$.

  \begin{lemma}
    \label{lem:nc_power_series}
    If $F = \sum_{w \in F_d^+} a_w x^w \in H^2_{nc}$ and $X \in \mathbb{B}_d^{n \times n}$, then
    \begin{equation*}
      \sum_{k=0}^\infty \Big\| \sum_{|w|=k} a_w X^w \Big\| < \infty.
    \end{equation*}
    Hence the series $F(X) := \sum_{k=0}^\infty \sum_{|w|=k} a_w X^w$ converges in $M_n(\mathbb{C})$.
    Moreover,
    \begin{equation*}
      \|F(X)\|^2 \le \frac{1}{1 - \|X\|^2} \|F\|_{H^2_{nc}}^2.
    \end{equation*}
  \end{lemma}

  \begin{proof}
    Let $r = \|X\| < 1$.
    For $k \in \mathbb{N}$, consider the tuple $(X^w)_{|w| = k}$.
    A little induction argument shows that
    for all $k \in \mathbb{N}$, the norm of this tuple (as an operator from $(\mathbb{C}^n)^{d^k}$ to $\mathbb{C}^n)$
    is at most $r^k$. We conclude that
    \begin{equation*}
      \Big\|\sum_{|w| = k} a_w X^w \Big\| \le \Big( \sum_{|w|=k} |a_w|^2 \Big)^{1/2} r^k,
    \end{equation*}
    since we may regard the sum on the left as a product of the row of  the $X^w$ and the column of the $a_w$.
    An application of the Cauchy--Schwarz inequality then completes the proof.
  \end{proof}

  Lemma \ref{lem:nc_power_series} shows that each $F \in H^2_{nc}$ defines a function
  on $\mathbb{B}_d^{nc}$ that takes $\mathbb{B}_d^{n \times n}$ into $M_n(\mathbb{C})$ for all $n \ge 1$.
  This function is a non-commutative function in the sense that is respects direct sums and similarities.
  Moreover, one can show that the coefficients of $F$ are uniquely determined by the values of $F$
  on $\mathbb{B}_d^{nc}$, so we can really think of $H^2_{nc}$ as a space of non-commutative functions
  on $\mathbb{B}_d^{nc}$.
  The study of non-commutative functions goes back to work of Taylor \cite{Taylor72,Taylor73}.
  For more background, see Part 2 of \cite{AMY20} and \cite{KV14}.

  The theory of non-commutative holomorphic functions on the ball and their relation to $H^2_{nc}$
  was in particular developed by Popescu \cite{Popescu06a}.
  It is also possible to regard $H^2_{nc}$ as a non-commutative RKHS in the sense of Ball, Marx and Vinnikov
  \cite{BMV16,BMV18}; see also Sections 2 and 3 and especially Proposition 3.2 in \cite{SSS18}.
  
  In particular, Lemma \ref{lem:nc_power_series} shows that if $F \in H^2_{nc}$, then $F(z)$ converges for each $z \in \mathbb{B}_d$.
  The connection between $H^2_{nc}$ and the Drury--Arveson space is given by the following result.
  \begin{proposition}
    \label{prop:NC_DA}
    The map
    \begin{equation*}
      R: H^2_{nc} \to H^2_d, \quad F \mapsto F \big|_{\mathbb{B}_d},
    \end{equation*}
    is a co-isometry.
  \end{proposition}

  \begin{proof}
  If $\lambda \in \mathbb{B}_d$, define
  \begin{equation*}
    k_\lambda = \sum_{w \in F_d^+} x^w \overline{\lambda}^w \in H^2_{nc}.
  \end{equation*}
  If $F \in H^2_{nc}$, then $F(\lambda) = \langle F,k_\lambda \rangle$.
  Moreover,
  \begin{align*}
    \langle k_\lambda, k_\mu \rangle_{H^2_{nc}}
    &= \sum_{n=0}^\infty \sum_{\substack{w \in F_d^+ \\ |w| = n}} \mu^w \overline{\lambda}^w = \sum_{n=0}^\infty \langle \mu,\lambda \rangle_{\mathbb{C}^d}^n \\
    &= \frac{1}{1 - \langle \mu,\lambda \rangle_{\mathbb{C}^d} }
    = \langle K_\lambda, K_\mu \rangle_{H^2_d},
  \end{align*}
  where $K_\lambda(z) = K(z,\lambda) = \frac{1}{1 - \langle z,\lambda \rangle }$ is the reproducing kernel of $H^2_d$.
  Hence, there exists an isometry 
  \begin{equation*}
    V: H^2_d \to H^2_{nc}, \quad V K_\lambda = k_\lambda \quad \text{ for all } \lambda \in \mathbb{B}_d.
  \end{equation*}
  If $F \in H^2_{nc}$, then
  \begin{equation*}
    F(\lambda) = \langle F,k_\lambda \rangle = \langle V^* F, K_\lambda \rangle = (V^* F)(\lambda).
  \end{equation*}
  So $R = V^*$ is a co-isometry.
  \end{proof}

  \begin{remark}
    \begin{enumerate}[label=\normalfont{(\alph*)},wide]
      \item 
  The last proof shows one very basic instance of how computations in the non-commutative
  setting may be easier than in the commutative setting.
  The non-commutative expansion of the reproducing kernel of $H^2_d$ as
  \begin{equation*}
    K(\mu,\lambda) = \sum_{w \in F_d^+} \mu^w \overline{\lambda}^w
  \end{equation*}
  is essentially trivial, whereas the commutative expansion as
  \begin{equation*}
    K(\mu,\lambda) = \sum_{\alpha \in \mathbb{N}^d} \frac{|\alpha|!}{\alpha!} \mu^\alpha \overline{\lambda}^\alpha
  \end{equation*}
  requires the multinomial theorem. Going from the non-commutative expansion
  to the commutative one amounts to counting how many words $w \in F_d^+$ correspond
  to a given multi-index $\alpha \in \mathbb{N}^d$.
\item
    One can also take Proposition \ref{prop:NC_DA} as the definition of $H^2_d$ by declaring that
    \begin{equation*}
      H^2_d = \{f: \mathbb{B}_d \to \mathbb{C}: \text{ there exists } F \in H^2_{nc} \text{ with } F \big|_{\mathbb{B}_d} = f\}
    \end{equation*}
    and
    \begin{equation*}
      \|f\|_{H^2_d} = \inf\{ \|F\|_{H^2_{nc}}: F \big|_{\mathbb{B}_d} = f \}.
    \end{equation*}
    \end{enumerate}
  \end{remark}
  
  Proposition \ref{prop:NC_DA} implies that we may identify $H^2_d$ with the image of $R^*$,
  which is a subspace of $H^2_{nc}$. What is this subspace? As we saw in the proof of Proposition \ref{prop:NC_DA},
  the operator $R^* = V$ is given by
  \begin{equation*}
    V: H^2_d \to H^2_{nc}, \quad V K_\lambda = k_\lambda =
    \sum_{w \in F_d^+} x^w \overline{\lambda}^w.
  \end{equation*}
  
  \begin{definition}
    A power series $F = \sum_{w \in F_d^+} a_w x^w$ is symmetric if $a_w = a_{w'}$
    whenever $w'$ is a permutation of $w$.
  \end{definition}

  \begin{example}
    If $d=2$, then $x_1 x_2$ is not symmetric, but $\frac{x_1 x_2 + x_2 x_1}{2}$ is.
    Each $k_\lambda$ for $\lambda \in \mathbb{B}_d$ is a symmetric power series.
    The element $x_1^2 x_2 + x_2^2 x_1$ it not symmetric.
  \end{example}

  \begin{proposition}
    $V$ defines a unitary from $H^2_d$ onto the space of symmetric power series in $H^2_{nc}$.
  \end{proposition}

  \begin{proof}
    It only remains to see that every symmetric power series belongs to the range of $V$.
    If $w \in F_d^+$, let $\alpha(w) \in \mathbb{N}^d$ be the multi-index
    obtained by counting the number of occurrences if each letter in $w$.
    Thus, $\alpha(w) =(\alpha_1,\ldots,\alpha_d)$, where $\alpha_i$ is the number of times the letter $i$ occurs in $w$. 
    
    Let $F = \sum_{w} a_w x^w \in H^2_{nc}$ be symmetric
    and suppose that $F$ is orthogonal to the range of $V$. We have to show that $F=0$.
    For all $\lambda \in \mathbb{B}_d$, we have
    \begin{equation*}
      0 = \langle F, k_\lambda \rangle = \sum_{w \in F_d^+} a_w \lambda^w = \sum_{\alpha \in \mathbb{N}^d} \Big( \sum_{w: \alpha(w) = \alpha} a_w \Big) \lambda^{\alpha(w)}.
    \end{equation*}
    It follows that $\sum_{w: \alpha(w) = \alpha} a_w = 0$ for all $\alpha \in \mathbb{N}^d$,
    and hence that $a_w = 0$ for all $w$ since $F$ is symmetric.
  \end{proof}

  \begin{remark}
    Using notation as in the last proof, note that
    \begin{equation*}
      k_\lambda = \sum_{w \in F_d^+} x^w \overline{\lambda}^w = \sum_{\alpha \in \mathbb{N}^d} \overline{\lambda}^{\alpha} \sum_{w: \alpha(w) = a} x^w
    \end{equation*}
    and
    \begin{equation*}
      K_\lambda = \sum_{w \in F_d^+} z^w \overline{\lambda}^w =  \sum_{\alpha \in \mathbb{N}^d} \frac{|\alpha|!}{\alpha!} \overline{\lambda}^\alpha z^\alpha,
    \end{equation*}
    since $| \{ w \in F_d^+: \alpha(w) = \alpha\}| = \frac{|\alpha|!}{\alpha!}$.
    Thus, using the identity $V K_\lambda = k_\lambda$ for $\lambda \in \mathbb{B}_d$ and comparing
    coefficients, we find that
  \begin{equation*}
    V z^\alpha = \frac{\alpha!}{|\alpha|!} \sum_{w: \alpha(w) = \alpha} x^w.
  \end{equation*}
  In other words, the image of a commutative monomial $z^\alpha$ is the average over all non-commutative
  monomials agreeing with $z^\alpha$ on scalars.
  This once again explains the appearance of the weights in the definition of the Drury--Arveson space.
  \end{remark}

  There is another closely related non-commutative point of view, which for instance appeared in Arveson's original paper \cite{Arveson98}.
  Let $E = \mathbb{C}^d$. The full Fock space over $E$ is
  \begin{equation*}
    \mathcal{F}(E) = \bigoplus_{n=0}^\infty E^{\otimes n}.
  \end{equation*}
  Denoting the standard orthonormal basis vectors of $E$ by $e_1,\ldots,e_n$, the following proposition
  is immediate from the definitions.

  \begin{proposition}
    There is a unique anti-unitary
    \begin{equation*}
      J: H^2_{nc} \to \mathcal{F}(E) \quad \text{ with } J( x_{i_1} x_{i_2} \ldots x_{i_r}) = e_{i_1} \otimes \ldots \otimes e_{i_r}.
    \end{equation*}
  \end{proposition}

  \begin{remark}
    One might also define $J$ to be unitary instead of anti-unitary.
    However, an argument can be made that it is more natural to identify $H^2_{nc}$ with the dual space of
    $\mathcal{F}(E)$, rather than with $\mathcal{F}(E)$ itself.
    For instance, in degree $n=1$, elements of $H^2_{nc}$ are homogeneous polynomials
    of degree one, i.e.\ linear forms on $E$.
  \end{remark}

  Symmetric elements of $H^2_{nc}$ correspond to symmetric tensors in $\mathcal{F}(E)$.
  This gives an identification of $H^2_d$ with the \emph{symmetric Fock space} $\mathcal{F}_+(E)$.
  
  \section{Multipliers and operator theory}

  \subsection{Multipliers}
  Whenever one considers a reproducing kernel Hilbert space, one is naturally led to studying its multipliers.
  In the case of the Drury--Arveson space, one might even say that
  the multipliers are really what give the space its place in the theory.

  \begin{definition}
    The \emph{multiplier algebra} of an RKHS $\mathcal{H}$ of functions on $X$ is
    \begin{equation*}
      \Mult(\mathcal{H}) = \Big\{ \varphi: X \to \mathbb{C}: \varphi \cdot f \in \mathcal{H} \text{ whenever } f \in \mathcal{H} \Big\}.
    \end{equation*}
    Elements of $\Mult(\mathcal{H})$ are called \emph{multipliers}.
    The \emph{multiplier norm} of a multiplier $\varphi$ is
    \begin{equation*}
      \|\varphi\|_{\Mult(\mathcal{H})} = \|M_\varphi: f \mapsto \varphi \cdot f\|_{B(\mathcal{H})}.
    \end{equation*}
  \end{definition}

  Since point evaluations are bounded, each multiplication operator $M_\varphi$
  has closed graph and is hence bounded by the closed graph theorem. Thus, $\|\varphi\|_{\Mult(\mathcal{H})} < \infty$
  for all $\varphi \in \Mult(\mathcal{H})$.
  For the Hardy space on $\mathbb{D}$, we have $\Mult(H^2) = H^\infty$, the algebra
  of all bounded holomorphic functions on $\mathbb{D}$, and $\|\varphi\|_{\Mult(H^2)} = \|\varphi\|_\infty$.

  \begin{example}
    \label{exa:z_i_mult}
    Arguably the most important multipliers of $H^2_d$ are the coordinate functions $z_1,\ldots,z_d$.
    To see that they are multipliers, we use the power series representation of $H^2_d$ to find that
    if $f = \sum_\alpha a_\alpha z^\alpha \in H^2_d$, then
    \begin{equation*}
      \|z_i f\|^2 = \sum_{\alpha \in \mathbb{N}^d} |a_\alpha|^2 \|z_i z^\alpha\|^2
      = \sum_{\alpha \in \mathbb{N}^d} |a_\alpha|^2 \frac{(\alpha + e_i)!}{(|\alpha| +1)!}
      \le \sum_{\alpha \in \mathbb{N}^d} |a_\alpha|^2 \frac{\alpha!}{|\alpha|!}
      = \|f\|^2,
    \end{equation*}
    so $z_i \in \Mult(H^2_d)$ with $\|z_i\|_{\Mult(H^2_d)} \le 1$.
    
  In fact, a small computation reveals that
  \begin{equation*}
    M_{z_i}^* z^\alpha =
    \begin{cases}
      \frac{\alpha_i}{|\alpha|} z^{\alpha - e_i} & \text{ if } \alpha_i \neq 0 \\
      0 & \text{ otherwise},
    \end{cases}
  \end{equation*}
  from which it follows that
  \begin{equation*}
    \sum_{i=1}^d M_{z_i} M_{z_i}^* = I - P_0,
  \end{equation*}
  where $P_0$ is the orthogonal projection onto $\mathbb{C}1$.
  Hence the row operator
  \begin{equation*}
    \begin{bmatrix}
      M_{z_1} & \cdots & M_{z_d}
    \end{bmatrix}: (H^2_d)^d \to H^2_d, \quad
    \begin{bmatrix}
      f_1 \\ \vdots \\ f_d
    \end{bmatrix}
    \mapsto \sum_{i=1}^n z_i f_i
  \end{equation*}
  is a contraction.
  \end{example}

  The last observation in the preceding example says that the tuple $M_z = (M_{z_1},\ldots,M_{z_d})$ is a \emph{row contraction}.
  We will see later that $M_z$ is even a universal row contraction in a suitable sense.

  We collect a few basic operator algebraic properties of $\Mult(H^2_d)$.
  It turns out to be convenient to work in the greater generality of RKHS.

  \begin{proposition}
    \label{prop:multiplier_basic}
    Let $\mathcal{H}$ be an RKHS on $X$ with $1 \in \mathcal{H}$.
    The map
    \begin{equation*}
      \Mult(\mathcal{H}) \to B(\mathcal{H}), \quad \varphi \mapsto M_\varphi,
    \end{equation*}
    is an injective unital homomorphism onto a weak operator topology closed subalgebra of $B(\mathcal{H})$.
    In particular, $\Mult(\mathcal{H})$ is a unital commutative Banach algebra.
  \end{proposition}

  \begin{proof}
    It is obvious that $\varphi \mapsto M_\varphi$ is a unital homomorphism;
    it is injective since $1 \in \mathcal{H}$.
    Let $k$ be the reproducing kernel of $\mathcal{H}$.
    WOT-closedness of
    \begin{equation*}
      \{M_\varphi: \varphi \in \Mult(\mathcal{H})\} \subset B(\mathcal{H})
    \end{equation*}
    will follow once we show that $T \in B(\mathcal{H})$ is a multiplication operator
    if and only if each kernel vector $k_w = k(\cdot,w)$ is an eigenvector of $T^*$.

    To see this fact, note that
    if $\varphi \in \Mult(\mathcal{H})$, then for all
    $w \in X$ and $f \in \mathcal{H}$, we have
    \begin{equation*}
      \langle f, M_\varphi^* k_w \rangle =
      \langle \varphi f, k_w \rangle = \varphi(w) f(w)
      = \langle f, \overline{\varphi(w)} k_w \rangle.
    \end{equation*}
    Hence $M_\varphi^* k_w = \overline{\varphi(w)} k_w$.

    Conversely, if $T \in B(\mathcal{H})$ has the property
    that $k_w$ is an eigenvector of $T^*$ for all $w \in X$,
    we define $\varphi: X \to \mathbb{C}$ by the identity
    \begin{equation*}
      T^* k_w = \overline{\varphi(w)} k_w \quad (w \in X).
    \end{equation*}
    (Note that $k_w \neq 0$ as $\langle 1, k_w \rangle = 1$.)
    Then
    \begin{equation*}
      (T f)(w) = \langle T f, k_w \rangle = \langle f, T^* k_w \rangle = \varphi(w) \langle f, k_w \rangle
      = \varphi(w) f(w)
    \end{equation*}
    for all $w \in X$ and $f \in \mathcal{H}$, so $T = M_\varphi$.
  \end{proof}

  Because of the proposition above, one usually identifies a multiplier with its multiplication operator.
  Thus, $\Mult(\mathcal{H})$ becomes a unital weak operator topology closed non-selfadjoint operator algebra
  in this way.

  \begin{remark}
    \phantomsection
    \label{rem:multipliers}
    \begin{enumerate}[label=\normalfont{(\alph*)},wide]
      \item The basic identity
        \begin{equation*}
          M_\varphi^* k_w = \overline{\varphi(w)} k_w,
        \end{equation*}
        established in the proof of Proposition \ref{prop:multiplier_basic}, forms
        the basis for many operator theoretic approaches to multipliers.
        We will discuss in Section \ref{sec:Pick} how it leads to a characterization
        of multipliers.
        Moreover, it shows how to recover the multiplier $\varphi$ from the multiplication
        operator $M_\varphi$, for instance as
        \begin{equation}
          \label{eqn:Berezin}
          \varphi(w) = \frac{\langle M_\varphi k_w,k_w \rangle }{\|k_w\|^2}.
        \end{equation}
        The expression on the right makes sense for more general operators $T \in B(\mathcal{H})$
        in place of $M_\varphi$, and is then called the \emph{Berezin transform} of $T$.

      \item
        If $\mathcal{A}$ is a subalgebra of some $B(\mathcal{H})$,
        one defines $\Lat(\mathcal{A})$ to be the lattice of all closed
        subspaces of $\mathcal{H}$ that are invariant under each $T \in \mathcal{A}$.
        If $\mathcal{N}$ is a collection of closed subspaces of $\mathcal{H}$,
        let $\Alg(\mathcal{N})$ be the algebra of all $T \in B(\mathcal{H})$ leaving
        each $M \in \mathcal{N}$ invariant. Then $\mathcal{A} \subset \Alg(\Lat(\mathcal{A}))$,
        and one says that $\mathcal{A}$ is \emph{reflexive} if $\mathcal{A} = \Alg(\Lat(\mathcal{A}))$;
        see \cite{Halmos71}.

        The proof of Proposition \ref{prop:multiplier_basic} shows that $\Mult(\mathcal{H})$
        is a reflexive operator algebra, since $T$ is a multiplication operator
        if and only if $T^*$ leaves $\mathbb{C} k_w$ invariant for every $w \in X$.
    \end{enumerate}
  \end{remark}

  Since $B(\mathcal{H})$ is the dual space of the trace class $T(\mathcal{H})$,
  it follows that $\Mult(\mathcal{H})$ is the dual space of $T(\mathcal{H}) / \Mult(\mathcal{H})_{\bot}$.
  In particular, this equips $\Mult(\mathcal{H})$ with a weak-$*$ topology.

  \begin{proposition}
    \label{prop:weak-star}
    On bounded subsets of $\Mult(\mathcal{H})$, the weak-$*$ topology agrees with the topology
    of pointwise convergence on $X$.
  \end{proposition}

  \begin{proof}
    Equation \eqref{eqn:Berezin} shows that weak-$*$ convergence implies pointwise convergence on $\mathcal{H}$.
    The converse follows from the identity
    \begin{equation*}
      \langle M_\varphi k_w, k_z \rangle = \varphi(z) k_w(z),
    \end{equation*}
    for $z,w \in \mathcal{H}$ and the fact that the linear span of the kernel functions in dense in $\mathcal{H}$.
  \end{proof}

  Specializing again to $H^2_d$, the following result of Davidson and Hamilton \cite[Corollary 5.3]{DH11} shows
  that weak-$*$ continuous functionals on $\Mult(H^2_d)$ admit a particularly simple representation.

  \begin{theorem}
    The algebra $\Mult(H^2_d)$ has property $\mathbb{A}_1(1)$, meaning that
    for every weak-$*$ continuous linear functional $L$ on $\Mult(H^2_d)$ with
    $\|L\| < 1$, there exist $f,g \in H^2_d$ with $\|f\| \|g\| < 1$
    and
    \begin{equation*}
      L(\varphi) = \langle \varphi f,g \rangle \quad \text{ for all } \varphi \in \Mult(H^2_d).
    \end{equation*}
    In particular, the weak-$*$ topology and the weak operator topology on $\Mult(H^2_d)$ agree.
  \end{theorem}

  It is worth remarking that the proof of this theorem crucially uses the non-commutative approach to the
  Drury--Arveson space; see Subsection \ref{ss:nc} and Subsection \ref{ss:nc_mult} below.

\subsection{Function theory of multipliers}
\label{ss:ft_mult}
  One of the key new features of the Drury--Arveson space in higher dimensions is that if $d \ge 2$,
  then the multiplier norm is no longer equal to the supremum norm on the ball.
  The function theoretic point of view taken in Subsection \ref{ss:ft_dA}, and especially the heuristic
  comparison to the Dirichlet space, give some indication
  that this might be the case.

  \begin{proposition}
    \label{prop:stric_containment}
    If $d \ge 2$, then $\Mult(H^2_d) \subsetneq H^\infty(\mathbb{B}_d)$,
    the inclusion is contractive,
    and the multiplier norm and the supremum norm on the ball are not equivalent.
  \end{proposition}

  \begin{proof}
    The contractive containment is a standard argument in the theory of reproducing kernels Hilbert spaces.
    If $\varphi \in \Mult(H^2_d)$, then $\varphi \in H^2_d \subset \mathcal{O}(\mathbb{B}_d)$. Moreover, 
    if $w \in \mathbb{B}_d$, then Remark \ref{rem:multipliers} shows that
    \begin{equation*}
      M_\varphi^* K_w = \overline{\varphi(w)} K_w,
    \end{equation*}
    where $K_w(z) = K(z,w)$ is the reproducing kernel of $H^2_d$, so
    $\overline{\varphi(w)}$ is an eigenvalue of $M_\varphi^*$.
    Hence $\|\varphi\|_\infty \le \|\varphi\|_{\Mult(H^2_d)}$.

    To see the strict containment, observe that by the arithmetic mean--geometric mean
    inequality, $\|(2 z_1 z_2)^n\|_{\infty} = 1$ for all $n \in \mathbb{N}$.
    On the other hand, the explicit formula for the norm of a monomial in $H^2_d$
    and Stirling's approximation show that
    \begin{equation}
      \label{eqn:Stirling}
      \|(2 z_1 z_2)^n\|^2_{\Mult(H^2_d)} \ge \|(2 z_1 z_2)^n\|_{H^2_d}^2
      = 4^n \frac{(n!)^2}{(2 n)!} \sim \sqrt{\pi n} \xrightarrow{n \to \infty} \infty.
    \end{equation}
    This shows that the multiplier norm and the supremum norm are not comparable.
    Moreover, the continuous inclusion $\Mult(H^2_d) \subset H^\infty(\mathbb{B}_d)$ is not
    bounded below, so it is not surjective by the open mapping theorem.
  \end{proof}

  Proposition \ref{prop:stric_containment} shows in particular that for $d \ge 2$, the multiplier algebra
  of the Drury--Arveson space is different from that of the Hardy space on the ball,
  since the latter algebra is $H^\infty(\mathbb{B}_d)$.

  \begin{remark}
    \begin{enumerate}[label=\normalfont{(\alph*)},wide]
      \item 
    The use of the open mapping theorem (or of one its variants) at the end of the last
    proof can easily be avoided, which then gives a fairly explicit example of a function in
    $H^\infty(\mathbb{B}_d) \setminus \Mult(H^2_d)$.
    For instance, following Arveson \cite[Theorem 3.3]{Arveson98}, we may define
    \begin{equation*}
      f(z) = \sum_{k=0}^\infty 2^{-k/4} (2 z_1 z_2)^{2^k}.
    \end{equation*}
    This sum converges absolutely in the supremum norm and hence
    $f$ even belongs to the ball algebra
    \begin{equation*}
      A(\mathbb{B}_d) = \{f \in C( \overline{\mathbb{B}_d}): f \big|_{\mathbb{B}_d} \text{ is holomorphic}\},
    \end{equation*}
    but $f \notin H^2_d$ by \eqref{eqn:Stirling}. Thus, we even find that $A(\mathbb{B}_d) \not\subset H^2_d$.
    
  \item The observation in (a) can be strengthened to the statement that there exists a function
    in $A(\mathbb{B}_d)$ that is not even the ratio of two functions in $H^2_d$;
    see \cite[Theorem 1.19]{AHM+ar}.

  \item
    It was pointed out by Fang and Xia \cite{FX19} that the existence of non-constant inner functions,
    due to Aleksandrov \cite{Aleksandrov82},
    gives another proof of $H^\infty(\mathbb{B}_d) \not\subset \Mult(H^2_d)$.
    Indeed, the oscillatory behaviour of non-constant inner functions near the boundary of the ball
    (which was known before their existence was known) implies that the gradient
    of a non-constant inner function does not belong to the Bergman space (see for instance \cite[19.1.4]{Rudin08}),
    hence neither does the radial derivative (see for instance \cite[Theorem 13]{ZZ08}).
    In combination with the function theoretical description of $H^2_d$ in Theorem \ref{thm:DA_FT},
    this gives that $H^2_d$ does not contain  any non-constant inner functions for $d \ge 2$.
    \end{enumerate}
  \end{remark}

  We will discuss in Section \ref{sec:Pick} a characterization of multipliers
  in terms of positivity of certain matrices.
  Nonetheless, it is natural to ask
  how to characterize multipliers of $H^2_d$ in function theoretic terms.
  Since the derivative appears in the function theoretic description of $H^2_d$,
  Carleson measure conditions for derivates will play a role, just as they do in the case of the Dirichlet space.
  To get a feeling for what is going on, it is again helpful to consider the case $d=2$.

  \begin{example}[Multipliers for $d=2$]
  Recall from Subsection \ref{ss:ft_dA} that $f \in H^2_2$ if and only if $R f \in L^2_a(\mathbb{B}_2)$
  and that
  \begin{equation}
    \label{eqn:norm_2}
    \|f\|^2_{H^2_2} \approx |f(0)|^2 + \int_{\mathbb{B}_2} |R f|^2 \, dV.
  \end{equation}
  We know from Proposition \ref{prop:stric_containment} that membership in $H^\infty(\mathbb{B}_2)$
  is a necessary condition for being a multiplier.
  Which other condition do we have to impose?
  
  Let $\varphi \in H^\infty(\mathbb{B}_2)$.
  Then $\varphi \in \Mult(H^2_2)$ if and only if $\varphi f \in H^2_2$ and
  \begin{equation*}
    \|\varphi f\|_{H^2_2}^2 \lesssim \|f\|_{H^2_2}^2 \quad \text{ for all } f \in H^2_2,
  \end{equation*}
  where the implied constant is independent of $f$.
  Since $\varphi \in H^\infty(\mathbb{B}_2)$, Equation \eqref{eqn:norm_2} implies that this holds if and only if
  $R(\varphi f)$ admits the estimate
  \begin{equation}
    \label{eqn:norm_2_2}
    \int_{\mathbb{B}_2} |R (\varphi f)|^2 \, d V \lesssim \|f\|^2_{H^2_2} \quad \text{ for all } f \in H^2_2.
  \end{equation}
  Now, the product rule shows that $R (\varphi f) = (R \varphi) f + \varphi R f$.
  Moreover, since $\varphi \in H^\infty(\mathbb{B}_2)$, the second summand always admits
  an estimate of the form
  \begin{equation*}
    \int_{\mathbb{B}_2} |\varphi R f|^2 d V \lesssim \int_{\mathbb{B}_2} |R f|^2 \,d V \lesssim
    \|f\|_{H^2_2}^2.
  \end{equation*}
  Therefore, by the triangle inequality in $L^2_a(\mathbb{B}_2)$, we see that
  $R(\varphi f)$ obeys the bound \eqref{eqn:norm_2_2} if and only if
  \begin{equation}
    \label{eqn:Carleson_measure}
    \int_{\mathbb{B}_2} |f|^2 |R \varphi|^2 \, d V \lesssim \|f\|^2_{H^2_2} \quad \text{ for all } f \in H^2_2.
  \end{equation}
  In the language of harmonic analysis, this means that $|R \varphi|^2 \, dV$ is a Carleson
  measure for $H^2_2$. Equivalently, $R \varphi$ is a multiplier from $H^2_2$ into $L^2_a(\mathbb{B}_2)$.
  This is precisely the additional condition we need to impose.

  To summarize, $\varphi \in \Mult(H^2_2)$ if and only if $\varphi \in H^\infty(\mathbb{B}_2)$
  and the Carleson measure condition \eqref{eqn:Carleson_measure} holds.
  \end{example}

  The reasoning above can be generalized to higher dimensions.
  There are additional complications because the product rule for higher derivatives
  involves more summands. It turns out that it is enough to control the first and the last
  summand in the expansion of $R^m(\varphi f)$, which leads again to a condition
  of the form ``$H^\infty$ + Carleson measure''.
  The precise result, due to Ortega and F\`abrega \cite{OF00}, is the following.
  \begin{theorem}
    \label{thm:ortega_fabrega}
    Let $\varphi: \mathbb{B}_d \to \mathbb{C}$. The following are equivalent:
    \begin{enumerate}[label=\normalfont{(\roman*)}]
      \item $\varphi \in \Mult(H^2_d)$;
      \item $\varphi \in H^\infty(\mathbb{B}_d)$ and for some (equivalently all) $m \in \mathbb{N}$  with $2 m -d > - 1$, there
        exists $C \ge 0$ such that
        \begin{equation*}
          \int_{\mathbb{B}_d} |f|^2 |R^m \varphi|^2 (1 - |z|^2)^{2 m -d} \, d V \le C \|f\|^2_{H^2_d} \quad \text{ for all } f \in H^2_d.
        \end{equation*}
    \end{enumerate}
  \end{theorem}
  Once again, in the language of harmonic analysis, the second condition above means that
  \begin{equation*}
    |R^m \varphi|^2 (1 - |z|^2)^{2 m - d} d V
  \end{equation*}
  is a Carleson measure for $H^2_d$.
  Proofs of this result can be found in \cite[Theorem 3.7]{OF00} (see also the later paper \cite{CFO10})
  and in \cite[Theorem 6.3]{AHM+18a}.

  To illustrate how the characterization in Theorem \ref{thm:ortega_fabrega} can be used,
  we consider the following example.

  \begin{example}
    Let $d=2$ and let $\varphi \in \Mult(H^2_2)$ with
    \begin{equation*}
      |\varphi(z)| \ge \varepsilon > 0 \quad \text{ for all } z \in \mathbb{B}_2.
    \end{equation*}
    Is $\frac{1}{\varphi} \in \Mult(H^2_2)$? Note that the corresponding question for $d=1$
    is trivial, as the multiplier algebra is $H^\infty$ in this case.
    
    For $d=2$, we may use the characterization in Theorem \ref{thm:ortega_fabrega}
    with $m=1$.
    Indeed, it is clear that $\frac{1}{\varphi} \in H^\infty(\mathbb{B}_2)$. To check
    the Carleson measure condition for $\frac{1}{\varphi}$, note that
    \begin{equation*}
      \Big| R \Big( \frac{1}{\varphi} \Big) \Big| = \Big| \frac{R \varphi}{\varphi^2} \Big|
      \le \frac{1}{\varepsilon^2} |R \varphi|.
    \end{equation*}
    Thus, the Carleson measure condition for $\varphi$ implies the Carleson measure condition for $\frac{1}{\varphi}$,
    so $\frac{1}{\varphi} \in \Mult(H^2_2)$.
  \end{example}

  The reasoning in the previous example can be extended to higher dimensions.
  Once again, additional work is required since higher order derivatives complicate matters.

  \begin{theorem}
    If $\varphi \in \Mult(H^2_d)$ with $|\varphi| \ge \varepsilon > 0$ on $\mathbb{B}_d$, then $\frac{1}{\varphi} \in \Mult(H^2_d)$.
  \end{theorem}
  This result can be seen as a very special case of the corona theorem for $H^2_d$
  due to Costea, Sawyer and Wick \cite{CSW11}, which we will discuss in Subsection \ref{ss:corona}.
  A direct proof of the result above was found by Fang and Xia \cite{FX13} and by Richter and Sunkes
  \cite{RS16}.

  To obtain a truly function theoretic characterization of multipliers from Theorem \ref{thm:ortega_fabrega},
  it remains to find a geometric characterization of Carleson measures.
  This was achieved by Arcozzi, Rochberg and Sawyer; see \cite{ARS08} for the precise statement.

  Since known function theoretic characterizations of multipliers can be difficult to work with in practice,
  it is also desirable to have good necessary and sufficient conditions.
  Fang and Xia considered the condition
  \begin{equation*}
    \sup_{w \in \mathbb{B}_d} \frac{\|\varphi K_w\|}{\|K_w\|} < \infty,
  \end{equation*}
  where $K_w(z) = \frac{1}{1 - \langle z,w \rangle }$ is the reproducing kernel of $H^2_d$,
  which is clearly necessary for $\varphi \in \Mult(H^2_d)$.
  In \cite{FX15}, they showed that it is not sufficient.

  The following sufficient condition was shown by Aleman, \mcc, Richter and the author.

  \begin{theorem}
    For $f \in H^2_d$, let
    \begin{equation*}
      V_f(z) = 2 \langle f, K_z f \rangle  - \|f\|^2 \quad (z \in \mathbb{B}_d).
    \end{equation*}
    If $\operatorname{Re} V_f$ is bounded in $\mathbb{B}_d$, then $f \in \Mult(H^2_d)$.
  \end{theorem}

  The proof can be found in \cite[Corollary 4.6]{AHM+17c}.
  The function $V_f$ in the theorem is called the \emph{Sarason function} of $f$.
  If $d=1$, then $\Re V_f$ is the Poisson integral of $|f|^2$.
  In particular, $f \in H^\infty$ if and only if $\Re V_f$ is bounded if $d=1$.
  Fang and Xia showed that if $d \ge 2$, then boundedness of $\Re V_f$
  is not necessary for $f \in \Mult(H^2_d)$, see \cite{FX20} and also
  \cite[Proposition 8.1]{AHM+20a}.

  \subsection{Dilation and von Neumann's inequality}
  \label{ss:dilation}

  As alluded to earlier, the tuple $M_z = (M_{z_1},\ldots,M_{z_d})$
  on the Drury--Arveson space plays a key role in multivariable operator.
  We first recall the relevant one variable theory.
  More background material on dilation theory can be found for instance
  in \cite{Paulsen02,Pisier01,Shalit21}.
  Throughout, we assume that $\mathcal{H}$ is a complex Hilbert space.

  The following fundamental result due to von Neumann \cite{Neumann51}
  forms the basis of a rich interplay between operator theory and function theory.

  \begin{theorem}[von Neumann's inequality]
    Let $T \in B(\mathcal{H})$ with $\|T\| \le 1$. For every polynomial $p \in \mathbb{C}[z]$,
    \begin{equation*}
      \|p(T)\| \le \sup \{ |p(z)| : |z| \le 1 \}.
    \end{equation*}
  \end{theorem}

  There are now many different proofs of von Neumann's inequality,
  see for instance \cite[Chapter 1]{Pisier01}, \cite[Chapters 1 and 2]{Paulsen02} and \cite{Drury83}.
  A particularly short proof uses the following dilation theorem due to Sz.-Nagy \cite{Sz.-Nagy53}.

  \begin{theorem}[Sz.-Nagy dilation theorem]
    Let $T \in B(\mathcal{H})$ with $\|T\| \le 1$. Then there exists a Hilbert space $\mathcal{K} \supset \mathcal{H}$
    and a unitary $U \in B(\mathcal{K})$ with
    \begin{equation*}
      p(T) = P_H p(U) \big|_H \quad \text{ for all } p \in \mathbb{C}[z].
    \end{equation*}
  \end{theorem}
  
  A dilation $U$ can be written down explicitly, see \cite{Schaeffer55}.

  Notice that Sz.-Nagy's dilation theorem reduces the proof of von Neumann's inequality
  to the case of unitary operators, which in turn easily follows from basic spectral
  theory as unitary operators are normal operators whose spectrum is contained in the unit circle.

  This is the general philosophy behind dilation theory: associate to a given operator on Hilbert space
  a better behaved operator on a larger Hilbert space.
  This idea had a profound impact on operator theory,
  see \cite{Paulsen02} and \cite{SFB+10}.

  There is a slight variant of the Sz.-Nagy dilation theorem,
  in which unitary operators are replaced by more general isometries,
  but the relationship between the original operator and its dilation becomes tighter.

  \begin{theorem}[Sz.-Nagy dilation theorem, second version]
    Let $T \in B(\mathcal{H})$ with $\|T\| \le 1$. Then there exists a Hilbert space $\mathcal{K} \supset \mathcal{H}$
    and an isometry $V \in B(\mathcal{K})$ such that $V^* \mathcal{H} \subset \mathcal{H}$ and
    \begin{equation*}
      T^* = V^* \big|_{\mathcal{H}}.
    \end{equation*}
  \end{theorem}
  
  Note that in the setting of this result, we in particular find that $p(T) = P_{\mathcal{H}} p(U) \big|_{\mathcal{H}}$ for all $p \in \mathbb{C}[z]$.
  A proof can be found for instance in \cite[Chapter 1]{Paulsen02}.
  It is not difficult to deduce the two versions of Sz.-Nagy's dilation theorem from each other.
  For instance, the first version can be obtained from the second by extending the isometry
  $V$ to a unitary $U$ on a larger Hilbert space,
  see again \cite[Chapter 1]{Paulsen02}.

  The conclusion of the first version of Sz.-Nagy's dilation theorem
  is often summarized by saying that ``every contraction dilates to a unitary'',
  whereas the second version says that ``every contraction co-extends to an isometry''.
  In terms of operator matrices, the relationship
  between $T$, a unitary dilation $U$ and an isometric co-extension $V$ can be understood as
  \begin{equation*}
    U =
    \begin{bmatrix}
      * & 0 & 0 \\
      * & T & 0 \\
      * & * & *
    \end{bmatrix}
  \end{equation*}
  (this is a result of Sarason \cite[Lemma 0]{Sarason65}) and
  \begin{equation*}
    V =
    \begin{bmatrix}
      T & 0 \\
      * & *
    \end{bmatrix},
  \end{equation*}
  see for instance \cite[Chapter 10]{AM02}.

  Let us now move into the multivariable realm.
  Let $T = (T_1,\ldots, T_d)$ be a tuple of commuting operators in $B(\mathcal{H})$:
  \begin{equation*}
    T_i T_j = T_j T_i \quad (i,j=1,\ldots,d).
  \end{equation*}

  Just as there is more than one reasonable extension of the unit disc to higher dimensions,
  there is more than one reasonable contractivity condition in multivariable operator theory.
  For instance, one might impose the contractivity condition $\|T_j\| \le 1$ for all $j$,
  i.e.\ one considers tuples of commuting contractions.
  The operator theory of these tuples connects to function theory on the polydisc $\mathbb{D}^d$.
  There is a large body of literature on commuting contractions.
  Very briefly, this theory works well for $d=2$, thanks to an extension of Sz.-Nagy's dilation
  theorem due to And\^o \cite{Ando63}.
  However, both von Neumann's inequality and And\^o's theorem fail for $d \ge 3$,
  which is a significant obstacle in the study of three or more commuting contractions,
  see \cite[Chapter 5]{Paulsen02}.

  Here, we will consider a different contractivity condition.
  \begin{definition}
    An operator tuple $T = (T_1,\ldots,T_d) \in B(\mathcal{H})^d$ is said to be a \emph{row contraction}
    if the row operator
    \begin{equation*}
          \begin{bmatrix}
            T_1 & \cdots & T_d
          \end{bmatrix}: \mathcal{H}^d \to \mathcal{H}, \quad (x_i)_{i=1}^d \mapsto \sum_{i=1}^d T_i x_i,
    \end{equation*}
    is a contraction.
  \end{definition}
  This is equivalent to demanding that
  \begin{equation*}
    \sum_{j=1}^d T_j T_j^* \le I.
  \end{equation*}
  Note that a tuple of scalars is a row contraction if and only if it belongs to the closed
  (Euclidean) unit ball.
  Thus, one expects that operator theory of commuting row contractions
  connects to function theory in the unit ball.
  This is indeed the case, and we will shortly see that row contractions
  are intimately related to the Drury--Arveson space.

  \begin{definition}
    A \emph{spherical unitary} is a tuple $U = (U_1,\ldots,U_d)$ of commuting normal operators with $\sum_{i=1}^d U_i U_i^* = I$.
  \end{definition}
  Notice that spherical unitaries in dimension one are simply unitaries.
  Spherical unitaries are essentially well understood thanks to the multivariable spectral theorem.
  In particular, every spherical unitary is unitarily equivalent
  to a direct sum of operator tuples of the form $M_z$ on $L^2(\mu)$,
  where $\mu$ is a measure supported on $\partial \mathbb{B}_d$,
  see \cite[Appendix D]{AM02} and \cite[Section II.1]{Davidson96}.

  It is not true that every commuting row contraction dilates to a spherical unitary.
  Indeed, the tuple $M_z$ on $H^2_d$ does not, since the multiplier norm on the Drury--Arveson space
  is not dominated by the supremum norm on the ball by Proposition \ref{prop:stric_containment},
  but spherical unitaries $U$ satisfy $\|p(U)\| \le \|p\|_\infty$ for all polynomials $p$.
  Instead, we have the following dilation theorem, which is
  one of the central results in the theory of the Drury--Arveson space.
  It explains its special place in multivariable operator theory.

  \begin{theorem}[Dilation theorem for $H^2_d$]
    \label{thm:DA_dilation}
    Let $T = (T_1,\ldots,T_d)$ be a commuting row contraction on $\mathcal{H}$.
    Then $T$ co-extends to a tuple of the form $S \oplus U$, where
    $U$ is spherical unitary and $S$ is a direct sum of copies of $M_z$ on $H^2_d$.

  \end{theorem}

    Explicitly, the co-extension statement means that there exist Hilbert spaces $\mathcal{K}$ and $\mathcal{E}$,
    an isometry $V: \mathcal{H} \to (H^2_d \otimes \mathcal{E}) \oplus \mathcal{K}$ and
    a spherical unitary $U$ on $\mathcal{K}$ such that
    \begin{equation*}
      ((M_{z_i} \otimes I_{\mathcal{E}}) \oplus U_i)^* V = V T_i^*
    \end{equation*}
    for $i=1,\ldots,d$. In this case, identifying $\mathcal{H}$ with a subspace of
    $(H^2_d \otimes \mathcal{E}) \oplus \mathcal{K}$ via the isometry $V$, we see
    that $\mathcal{H}$ is invariant under $((M_{z_i} \otimes I_{\mathcal{E}}) \oplus U_i)^*$
    and that
    \begin{equation*}
      T_i^* = ((M_{z_i} \otimes I_{\mathcal{E}}) \oplus U_i)^* \big|_{\mathcal{H}}.
    \end{equation*}
    In particular,
    \begin{equation*}
      p(T) = P_{\mathcal{H}} p( (M_z \otimes I) \oplus U) \big|_{\mathcal{H}} \quad \text{ for all } p \in \mathbb{C}[z_1,\ldots,z_d].
    \end{equation*}

    Various versions of Theorem \ref{thm:DA_dilation} were proved by Drury \cite{Drury78},
  by M\"uller and Vasilescu \cite{MV93} and by Arveson \cite{Arveson98}.

  Notice that the case $d=1$ recovers the second version of Sz.-Nagy's dilation theorem.
  Indeed, in dimension one, every operator of the form $S \oplus U$ is clearly an isometry,
  and these are in fact all isometries by the Wold decomposition,
  see for example \cite[Theorem V.2.1]{Davidson96}.
  If $d \ge 2$, then there is no direct analogue of the first version of Sz.-Nagy's dilation theorem.
  This is closely related to the fact that the identity representation of the algebra
  generated by $M_z$ is a boundary representation in the sense of Arveson; see Lemma 7.13 and the discussion
  following it in \cite{Arveson98}.

  Just as in one variable, the dilation theorem implies an inequality of von Neumann type.
  \begin{corollary}[Drury's von Neumann inequality]
    Let $T = (T_1,\ldots,T_d)$ be a commuting row contraction. Then
    \begin{equation*}
      \|p(T)\| \le \|p\|_{\Mult(H^2_d)} \quad \text{ for all } p \in \mathbb{C}[z_1,\ldots,z_d].
    \end{equation*}
  \end{corollary}
  This inequality makes clear the importance of the multiplier norm on the Drury--Arveson space.
  Note that the inequality is sharp, since we obtain equality by taking $T =M_z$ on $H^2_d$.

  Drury's von Neumann inequality can be deduced from Theorem \ref{thm:DA_dilation} by using
  that spherical unitaries $U$ satisfy $\|p(U)\| \le \|p\|_\infty \le \|p\|_{\Mult(H^2_d)}$
  for all polynomials $p$. We will give a more direct argument.

  The special case of Theorem \ref{thm:DA_dilation} in which the spherical unitary summand is absent
  is important. The key concept is that of \emph{purity}.
  Let $T = (T_1,\ldots,T_d)$ be a row contraction.
  Let
  \begin{equation*}
    \theta: B(\mathcal{H}) \to B(\mathcal{H}), \quad A \mapsto \sum_{i=1}^d T_i A T_i^*.
  \end{equation*}
  Then $\theta(I) \le I$, and so $I \ge \theta(I) \ge \theta^2(I) \ge \ldots \ge 0$.
  We say that $T$ is \emph{pure} if $\lim_{n \to \infty} \theta^n(I) = 0$ in the strong operator topology.
  In the pure case, we obtain the following version of the dilation theorem.

  \begin{theorem}
    \label{thm:DA_dilation_pure}
    Every pure commuting row contraction co-extends to a direct sum of copies of $M_z$ on $H^2_d$.
  \end{theorem}

  This result is sufficient for proving Drury's von Neumann inequality.
  \begin{proof}[Proof of Drury's von Neumann inequality]
  If $T$ is a row contraction, then $r T$ is a pure row contraction for $0 \le r < 1$.
  Indeed,
  \begin{equation*}
    \theta(I) \le r^2 I
  \end{equation*}
  and so
  \begin{equation*}
    \theta^n(I) \le r^{2 n} I,
  \end{equation*}
  which even converges to zero in norm as $n \to \infty$.
  
  Now, if $p \in \mathbb{C}[z]$, then Theorem \ref{thm:DA_dilation_pure} implies that
  \begin{equation*}
    \|p(rT)\| \le \|p(M_z)\| = \|p\|_{\Mult(H^2_d)}.
  \end{equation*}
  Now let $r \nearrow 1$.
  \end{proof}

  We now provide a proof of the dilation theorem in the pure case.

  \begin{proof}[Proof of Theorem \ref{thm:DA_dilation_pure}]
    Since $T$ is a row contraction, we may define
  \begin{equation*}
  \Delta = \Big(I - \sum_{i=1}^d T_i T_i^* \Big)^{1/2} = (I - \theta(I))^{1/2}.
  \end{equation*}
  Let
  \begin{equation*}
    V: \mathcal{H} \to H^2_d \otimes \mathcal{H},
    \quad V h = \sum_{\alpha \in \mathbb{N}^d} \frac{|\alpha|!}{\alpha!} z^\alpha \otimes \Delta (T^*)^\alpha h.
  \end{equation*}
  Using purity, we find that
  \begin{align*}
    \|V h\|^2 = \sum_{\alpha \in \mathbb{N}^d} \frac{|\alpha|!}{\alpha!} \langle T^\alpha \Delta^2 (T^*)^\alpha h, h \rangle 
    &= \lim_{N \to \infty} \sum_{n=0}^N \langle (\theta^n(I) - \theta^{n+1}(I)) h, h \rangle  \\
    &= \|h\|^2 - \lim_{N \to \infty} \langle \theta^{N+1}(I) h, h \rangle =  \|h\|^2.
  \end{align*}
  Hence $V$ is an isometry.

  Finally, the formula for $M_{z_i}^*$ in Example \ref{exa:z_i_mult} shows that
  \begin{equation*}
    M_{z_i}^* \frac{|\alpha|!}{\alpha!} z^\alpha = \frac{|\alpha - e_i|!}{(\alpha - e_i)!} z^{\alpha - e_i}
  \end{equation*}
  if $\alpha_i \neq 0$, so
  \begin{equation*}
    (M_{z_i}^* \otimes I) V = V T_i^*.
  \end{equation*}
  Thus, $\ran(V)$ is invariant under $M_{z_i}^* \otimes I$
  and $M_{z_i}^* \otimes I \big|_{\ran(V)}$ is unitarily equivalent to $T_i^*$.
  \end{proof}

  We will give a proof of the dilation theorem in the non-pure case in Subsection \ref{ss:toeplitz}.

  \begin{remark}
    Since the multiplier norm of $H^2_d$ and the supremum norm are not comparable
    if $d \ge 2$ (Proposition \ref{prop:stric_containment}), the right-hand side
    in Drury's von Neumann inequality cannot be replaced with a constant times the supremum norm.
    One might nonetheless ask if this is possible if one only considers row contractive matrix tuples
    of a fixed matrix size. It was shown by Richter, Shalit and the author that this can be done:
    For every $d,n \in \mathbb{N}$, there exists a constant $C(d,n)$ so that
    for every row contraction $T$ consisting of $d$ commuting $n \times n$ matrices and every polynomial
    $p$, the inequality
    \begin{equation*}
      \|p(T)\| \le C(d,n) \sup_{z \in \mathbb{B}_d} |p(z)|
    \end{equation*}
    holds. In fact, one can even take $C_{d,n}$ to be independent of $d$;
    see \cite{HRS21}.
  \end{remark}

  \subsection{The non-commutative approach to multipliers}
  \label{ss:nc_mult}

  The approach to the Drury--Arveson space using non-commutative functions outlined in Subsection \ref{ss:nc}
  is also very useful in the context of multipliers.
  Define
  \begin{equation*}
    H^\infty_{nc} = \Big\{\Phi \in H^2_{nc}: \sup_{X \in \mathbb{B}_d^{nc}}: \|\Phi(X)\| < \infty \Big\}.
  \end{equation*}
  Each $\Phi \in H^\infty_{nc}$ defines a left multiplication operator
  \begin{equation*}
    L_\Phi: H^2_{nc} \to H^2_{nc}, \quad F \mapsto \Phi F,
  \end{equation*}
  and $\|L_\Phi\| = \sup_{X \in \mathbb{B}_d^{nc}} \|\Phi(X)\|$.
  In fact, every left multiplication operator is of this form; see \cite[Section 3]{SSS18} for more details. 
  Thus, we think of $H^\infty_{nc}$ as a non-commutative analogue of $\Mult(H^2_d)$.

  Of particular importance are the left multiplication operators by the variables $L_{i} := L_{x_i}$.
  These are isometries with pairwise orthogonal ranges.
  The space $\{L_\Phi: \Phi \in H^\infty_{nc}\}$ turns out to be the WOT-closed algebra generated
  by $L_1,\ldots,L_d$. This algebra is known as the non-commutative analytic Toeplitz algebra $\mathcal{L}_d$
  and has been the subject of intense study, in particular by Davidson and Pitts \cite{DP98a,DP98,DP99},
  by Popescu \cite{Popescu91,Popescu95,Popescu06a} and by Arias and Popescu \cite{AP00} , to only name a few references. It is also one of the Hardy algebras
  of Muhly and Solel \cite{MS04}.

  There is a version of the dilation theorem for $H^2_d$ in the non-commutative context, due
  to Bunce \cite{Bunce84}, Frazho \cite{Frazho82} and Popescu \cite{Popescu89a,Popescu89}.

  \begin{theorem}
    Every (not necessarily commuting) row contraction $T$ co-extends to a tuple of isometries
    with pairwise orthogonal ranges. If $T$ is pure, then the co-extension
    can be chosen to be a direct sum of copies of the tuple $(L_1,\ldots,L_d)$ on $H^2_{nc}$.
  \end{theorem}

  The connection between $H^\infty_{nc}$ and $\Mult(H^2_d)$ is given by the following result
  of Davidson and Pitts; see Corollary 2.3 and Section 4 in \cite{DP98}.
  It is an analogue of the (much simpler) Hilbert space result Proposition \ref{prop:NC_DA}.

  \begin{theorem}
    \label{thm:NC_DA_mult}
    The map
    \begin{equation*}
      R: H^\infty_{nc} \to \Mult(H^2_d), \quad \Phi \mapsto \Phi \big|_{\mathbb{B}_d},
    \end{equation*}
    is a complete quotient map. If $\varphi \in \Mult(H^2_d)$, then there
    exists $\Phi \in H^\infty_{nc}$ with $\Phi \big|_{\mathbb{B}_d} = \varphi$ and $\|\Phi\|_{H^\infty_{nc}} = \|\varphi\|_{\Mult(H^2_d)}$.
  \end{theorem}

  Quotient map means that $R$ maps the open unit ball of $H^\infty_{nc}$ onto the open unit ball of $\Mult(H^2_d)$,
  and the modifier ``completely'' means that the same property holds
  for all induced maps $M_n(H^\infty_{nc}) \to M_n(\Mult(H^2_d))$, defined by applying $R$ entrywise.
  In particular, complete quotient maps are completely contractive.
  For background on these notions, see \cite{Paulsen02}.
  The fact that every multiplier has a norm preserving lift can be deduced from the commutant
  lifting theorem for $H^2_{nc}$ of Frazho \cite{Frazho84} and Popescu \cite{Popescu89a};
  see also \cite{DL10}.

  There are questions that turn out to be simpler in the non-commutative setting of $H^2_{nc}$ than in the commutative
  setting of $H^2_d$.
  For instance, one of the advantages in the non-commutative setting is the fact that the operators $L_i$ are isometries with pairwise orthogonal ranges.
  This leads to a powerful Beurling theorem, due to Arias and Popescu \cite[Theorem 2.3]{AP00} and Davidson and Pitts \cite[Theorem 2.1]{DP99}.
  Jury and Martin used this non-commutative Beurling theorem to answer open questions in the commutative setting;
  see \cite{JM18,JM18a}.
  The passage from the non-commutative to the commutative setting is made possible by Proposition \ref{prop:NC_DA}
  and Theorem \ref{thm:NC_DA_mult}.

  \subsection{The Toeplitz algebra}
  \label{ss:toeplitz}
  In his proof of the dilation theorem in the non-pure case, Arveson
  made use of the Toeplitz $C^*$-algebra $\mathcal{T}_d$, which is defined to be
  the unital $C^*$-subalgebra of $B(H^2_d)$ generated by $M_{z_1},\ldots,M_{z_d}$.

  A key result about $\mathcal{T}_d$ is the following theorem, again due to Arveson, which is \cite[Theorem 5.7]{Arveson02}.
  We let $\mathcal{K}$ denote the ideal of compact operators in $B(H^2_d)$.

  \begin{theorem}
    \label{thm:Toeplitz_short_exact}
    There is a short exact sequence of $C^*$-algebras
    \begin{equation*}
      0 \longrightarrow \mathcal{K} \longrightarrow \mathcal{T}_d \longrightarrow C(\partial \mathbb{B}_d) \longrightarrow 0,
    \end{equation*}
    where the first map is the inclusion and the second map sends $M_{z_i}$ to $z_i$.
  \end{theorem}

  \begin{proof}
    As remarked in Example \ref{exa:z_i_mult}, we have
    \begin{equation}
      \label{eqn:rank_one}
      \sum_{i=1}^d M_{z_i} M_{z_i}^* = I - P_0,
    \end{equation}
    where $P_0$ denotes the orthogonal projection onto the constant functions.
    Hence $P_0 \in \mathcal{T}_d$. A little computation shows that if $p,q$ are polynomials
    in $d$ variables, then
    \begin{equation*}
      (M_p P_0 M_q^*)(f) = \langle f,q \rangle p \quad \text{ for all }  f \in H^2_d.
    \end{equation*}
    Hence $\mathcal{T}_d$ contains all rank one operators of the form $\langle \cdot,q \rangle p$
    for polynomials $p$ and $q$, and therefore $\mathcal{K} \subset \mathcal{T}_d$.
    
    It remains to identify the quotient $\mathcal{T}_d / \mathcal{K}$.
    The formula for the action $M_{z_i}^*$ on monomials in Example \ref{exa:z_i_mult} shows that
    \begin{equation*}
      (M_{z_i}^* M_{z_i} - M_{z_i} M_{z_i}^*) z^\alpha = \Big( \frac{\alpha_i+1}{|\alpha|+1} - \frac{\alpha_i}{|\alpha|} \Big) z^\alpha = \frac{|\alpha| - \alpha_i}{ |\alpha| (|\alpha| + 1)} z^\alpha.
    \end{equation*}
    for all $\alpha \in \mathbb{N}^d \setminus \{0\}$.
    Hence, with respect to the orthogonal basis of monomials, the operator $M_{z_i}^* M_{z_i} - M_{z_i} M_{z_i}^*$
    is a diagonal operator whose diagonal tends to zero, and so it is compact.
    Therefore, the quotient $C^*$-algebra $\mathcal{T}_d / \mathcal{K}$ is generated by the $d$ commuting
    normal elements $N_i = M_{z_i} + \mathcal{K}$ and so it is commutative by Fuglede's theorem; see
    \cite[Theorem 12.16]{Rudin91}. (Alternatively, one can compute directly that $M_{z_j}^* M_{z_i} - M_{z_i} M_{z_j}^*$ is compact for all $i,j = 1,\ldots,d$).
    
    By the Gelfand--Naimark theorem (see \cite[Section I.3]{Davidson96} or
    \cite[Section 2.2]{Arveson02}), the commutative $C^*$-algebra $\mathcal{T}_d/\mathcal{K}$
    is isomorphic to $C(K)$ for some compact Hausdorff space $K$. More precisely,
    if
    \begin{equation*}
      K = \{ (\chi(N_1), \ldots, \chi(N_d)): \chi \text{ character of } \mathcal{T}_d / \mathcal{K} \} \subset \mathbb{C}^d,
    \end{equation*}
    then the maximal ideal space of $\mathcal{T}_d / \mathcal{K}$ is identified with $K$,
    and modulo this identification, the Gelfand transform takes the form
    \begin{equation*}
      \mathcal{T}_d / \mathcal{K} \to C(K), \quad M_{z_i} + \mathcal{K} \mapsto z_i,
    \end{equation*}
    so this map is a $*$-isomorphism.

    Finally, from \eqref{eqn:rank_one}, it follows that $K$ is a non-empty subset of $\partial \mathbb{B}_d$.
    Unitary invariance of $H^2_d$ (Proposition \ref{prop:unit_invariance}) then implies
    that $K = \partial \mathbb{B}_d$, which gives the result.
  \end{proof}

  We now discuss how the dilation theorem in the pure case (Theorem \ref{thm:DA_dilation_pure})
  and Theorem \ref{thm:Toeplitz_short_exact} can be used to prove the dilation theorem in general (Theorem \ref{thm:DA_dilation}). The starting point is the observation that if $T$ is a row contraction, then $r T$ is a pure row contraction for all $0 \le r < 1$, see the proof of Drury's von Neumann inequality.
  Thus, the dilation theorem in the pure case yields for each $ r < 1$ a dilation of $r T$,
  and one would like to take some kind of limit as $r \nearrow 1$.
  Doing this directly is tricky, as the individual dilations are difficult to control.

  There is a different approach due to Arveson \cite{Arveson72}, which in general relies on Arveson's
  extension theorem \cite{Arveson69} and Stinespring's dilation theorem \cite{Stinespring55}.
  A comprehensive exposition of this technique can be found in \cite[Chapter 7]{Paulsen02} and \cite[Section 7]{Shalit21}. Briefly, Arveson showed that the existence of dilations is equivalent to certain maps
  being completely contractive or completely positive, and these behave well with respect to taking
  limits. For instance, we saw that Sz.-Nagy's dilation theorem implies von Neumann's inequality.
  Arveson's technique allows us to go back and deduce Sz.-Nagy's dilation theorem from von Neumann's inequality (for matrix polynomials).

  In our setting, the argument runs as follows.

  \begin{proof}[Proof of Theorem \ref{thm:DA_dilation}]
    Let $T$ be a commuting row contraction on $\mathcal{H}$ and let
    \begin{equation*}
      \mathcal{S} = \spa \{ M_z^\alpha (M_z^\beta)^*: \alpha, \beta \in \mathbb{N}^d \} \subset \mathcal{T}_d.
    \end{equation*}
    For each $0 \le r < 1$, the tuple $r T$ is a pure row contraction,
    so by Theorem \ref{thm:DA_dilation_pure}, there exists an isometry $V_r: \mathcal{H} \to H^2_d \otimes \mathcal{E}_r$ such that
    \begin{equation}
      \label{eqn:dilation_r}
      (M_{z_i} \otimes I)^* V_r = V_r r T_i^*
    \end{equation}
    for $i=1,\ldots,d$. Define
    \begin{equation*}
      \varphi_r: \mathcal{S} \to B(\mathcal{H}), \quad X \mapsto V_r^* X V_r.
    \end{equation*}
    Then $\varphi_r$ is a unital and completely positive, meaning
    that it maps positive operators to positive operators, and the same is true
    for all maps $M_n(\mathcal{S}) \to M_n(B(\mathcal{H}))$ defined by applying $\varphi_r$ entrywise.

    Moreover \eqref{eqn:dilation_r} implies that
    \begin{equation*}
      \varphi_r(M_z^\alpha (M_z^\beta)^*) = (r T)^\alpha ((r T)^\beta)^*
    \end{equation*}
    Hence, for each $X \in \mathcal{S}$, the limit
    $\varphi(X) := \lim_{r \to 1} \varphi_r(X)$ exists, and this defines a unital
    completely positive map $\varphi: \mathcal{S} \to B(\mathcal{H})$
    with
    \begin{equation*}
      \varphi( M_z^\alpha (M_z^\beta)^*) = T^\alpha (T^\beta)^* \quad \text{ for all } \alpha, \beta \in \mathbb{N}^d.
    \end{equation*}

    By Arveson's extension theorem and Stinespring's dilation theorem (see \cite[Theorem 7.5 and Theorem 4.1]{Paulsen02}), there exist a Hilbert space $\mathcal{L}$, an isometry $V: \mathcal{H} \to \mathcal{L}$ and a unital $*$-homomorphism $\pi: \mathcal{T}_d \to B(\mathcal{L})$ such that
    \begin{equation*}
      V^* \pi( M_z^\alpha (M_z^\beta)^*) V = T^\alpha (T^\beta)^* \quad \text{ for all } \alpha, \beta \in \mathbb{N}^d.
    \end{equation*}
    We claim that the range of $V$ is invariant under $\pi(M_{z_i})^*$ for each $i$.
    To see this, note that
    \begin{equation*}
      V^* \pi(M_{z_i} M_{z_i}^*) V = T_i T_i^* = V^* \pi(M_{z_i}) V V^* \pi(M_{z_i})^* V
    \end{equation*}
    and so
    \begin{align*}
      0 &= V^* \pi(M_{z_i}) (I - V V^*) \pi(M_{z_i})^* V \\
      &= ( (I - V V^*) \pi(M_{z_i})^* V)^*
      ( (I - V V^*) \pi(M_{z_i})^* V),
    \end{align*}
    hence $(I - V V^*) \pi(M_{z_i})^* V = 0$. This
    shows invariance of the range of $V$ under $\pi(M_{z_i})^*$. Thus,
    \begin{equation*}
      \pi(M_{z_i})^* V = V V^* \pi(M_{z_i})^* V = V T_i^*.
    \end{equation*}

    It remains to show that the tuple $\pi(M_z)$ is unitarily equivalent to $S \oplus U$,
    where $S$ is a direct sum of copies of $M_z$ and $U$ is spherical unitary.
    To this end, we use the structure of the Toeplitz $C^*$-algebra observed in Theorem \ref{thm:Toeplitz_short_exact}
    and basic representation theory of $C^*$-algebras (see the discussion preceding Theorem I.3.4 in
    \cite{Arveson76}) to conclude that $\pi: \mathcal{T}_d \to B(\mathcal{L})$
    splits as a direct sum $\pi = \pi_1 \oplus \pi_2$.
    Here, $\pi_1$ is induced by a representation of $\mathcal{K}$ and hence
    is unitarily equivalent to a multiple of the identity representation of $\mathcal{T}_d$ on $B(H^2_d)$,
    and $\pi_2$ factors through the quotient $\mathcal{T}_d / \mathcal{K} \cong C(\partial \mathbb{B}_d)$.
    Thus, $\pi_1(M_z)$ is unitarily equivalent to a direct sum of copies of $M_z$,
    and $\pi_2(M_z)$ is spherical unitary, as desired.
  \end{proof}

  \begin{remark}
    \begin{enumerate}[label=\normalfont{(\alph*)},wide]
      \item Examination of the proof of Theorem \ref{thm:Toeplitz_short_exact} shows that
        the operator system $\mathcal{S}$ used in the proof of Theorem \ref{thm:DA_dilation}
        is dense in the $C^*$-algebra $\mathcal{T}_d$.
        Thus, the use of Arveson's extension theorem  could be avoided.
        In turn, this leads to a uniqueness statement about co-extensions
        of the form $S \oplus U$;  see \cite[Theorem 8.5]{Arveson98} for details.
      \item
        It was remarked before the proof of Theorem \ref{thm:DA_dilation} that Arveson's technique
        makes it possible to deduce dilation theorems from von Neumann type inequalities.
        This can also be done here.
        Drury's von Neumann inequality shows that the map
        \begin{equation*}
          \mathcal{T}_d \supset \spa \{M_z^\alpha: \alpha \in \mathbb{N}^d\} \to B(\mathcal{H}), \quad M_z^\alpha \mapsto T^\alpha,
        \end{equation*}
        is contractive, and the same proof shows that it is in fact completely contractive.
        Applying Arveson's extension theorem and Stinespring's dilation theorem
        to this map and using Theorem \ref{thm:Toeplitz_short_exact}
        yields a dilation of $T$ of the form $S \oplus U$.
        The advantage of working with the larger operator system $\mathcal{S}$ in the proof is
        that we not only obtain a dilation, but a co-extension.
        This idea already appeared in work of Agler \cite{Agler82}.
        For the relationship between dilations and co-extensions in the context of reproducing
        kernel Hilbert spaces, the reader is also referred to \cite{CH18}.
      \item
  The proof of the dilation theorem in the non-pure case of M\"uller and Vasilescu, see \cite[Theorem 11]{MV93}, does not
  rely on the representation theory of the Toeplitz $C^*$-algebra.
  Instead, M\"uller and Vasilescu construct by hand an extension of $T^*$ of the form $(M_z \otimes I)^* \oplus V$,
  where $V$ is a spherical isometry, i.e.\ $\sum_{i=1}^d V_i^* V_i$.
  They then use an earlier result of Athavale \cite{Athavale90} to extend the spherical
  isometry to a spherical unitary.

  Another proof of the dilation theorem was given by Richter and Sundberg \cite{RS10}.
  Their proof uses Agler's theory of families and extremal operator tuples \cite{Agler88}.
    \end{enumerate}
  \end{remark}

  The short exact sequence in Theorem \ref{thm:Toeplitz_short_exact} encodes in particular the basic fact
  that the tuple $M_z$ on $H^2_d$ consists of essentially normal operators,
  i.e. that $M_{z_i}^* M_{z_i} - M_{z_i} M_{z_i}^*$ is a compact operator
  for each $1 \le i \le d$.

  Let $I \subset \mathbb{C}[z_1,\ldots,z_d]$ be a homogeneous ideal.
  To avoid certain trivialities, we assume $I$ to be of infinite co-dimension.
  By Hilbert's Nullstellensatz, this is equivalent to demanding that the vanishing
  locus
  \begin{equation*}
    V(I) = \{z \in \mathbb{C}^d: p(z) = 0 \text{ for all } p \in I \}
  \end{equation*}
  does not consist only of the origin.
  On $I^\bot = H^2_d \ominus I$, we define the operator tuple $S^I = (S_1,\ldots,S_d)$
  by
  \begin{equation*}
    S_i = P_{I^\bot} M_{z_i} \big|_{I^\bot}.
  \end{equation*}
  Since $I$ is invariant under $M_z$, the tuple $S^I$ is still
  a tuple of commuting operators.
  A famous conjecture of Arveson asserts that $S^I$ should also be essentially
  normal. In fact, Arveson made a stronger conjecture \cite{Arveson02a},
  which was further refined by Douglas \cite{Douglas06} in the following form.

  \begin{conjecture}[Arveson--Douglas essential normality conjecture]
    The commutators $S_j S_k^* - S_k^* S_j$ belong to the Schatten class
    $\mathcal{S}^p$ for all $p > \dim V(I)$ and all $1 \le j,k \le d$.
  \end{conjecture}

  Arveson's initial motivation came from his work on the what is known as the curvature invariant
  \cite{Arveson02a}, but there are other reasons for considering this conjecture.
  For instance, if the conjecture is true, then letting $\mathcal{T}_I$ denote
  the $C^*$-algebra generated by $S$, one obtains a short exact sequence
  \begin{equation*}
    0 \longrightarrow \mathcal{K} \longrightarrow \mathcal{T}_I \longrightarrow C( V(I) \cap \partial \mathbb{B}_d) \longrightarrow 0,
  \end{equation*}
  see for instance \cite[Section 5]{GW08}.
  This would be analogous to the short exact sequence in Theorem \ref{thm:Toeplitz_short_exact}.
  The vision of Arveson and Douglas was, very roughly speaking, to connect operator theory
  of the tuple $S^I$ to algebraic geometry of the variety $V(I)$.

  The conjecture remains open in general, but it has been verified in special cases.
  In particular, Guo and Wang showed the following result;
  see Proposition 4.2 and Theorem 2.2 in \cite{GW08}, see also \cite[Theorem 2.3]{Eschmeier11} and 
  \cite[Theorem 4.2]{Shalit11}.

  \begin{theorem}
    The Arveson--Douglas conjecture holds for $S^I$ in each of the following cases:
    \begin{enumerate}[label=\normalfont{(\alph*)}]
      \item $d \le 3$;
      \item $I$ is a principal homogeneous ideal.
    \end{enumerate}
  \end{theorem}

  We only mention one more result, which was independently
  shown by Engli\v{s} and Eschmeier \cite{EE15} and by Douglas, Tang and Yu \cite{DTY16}.

  \begin{theorem}
    If $I$ is a radical ideal (i.e. $I = \{p: p \big|_{V(I)} = 0 \}$) and $V(I)$ is smooth away from $0$,
    then the Arveson--Douglas conjecture holds for $S^I$.
  \end{theorem}

  For a thorough discussion of the Arveson--Douglas essential normality conjecture,
  see \cite[Section 10]{Shalit13} and \cite[Section 5]{FX19}.

\subsection{Functional calculus}

  If $T$ is a tuple of commuting operators, then there is no issue in making sense
  of $p(T)$ for a polynomial $p \in \mathbb{C}[z_1,\ldots,z_d]$.
  However, it is often desirable to extend the supply of functions $f$ for which one can make sense of $f(T)$;
  this is called a functional calculus for $T$.

  Let $\mathcal{A}_d$ be the multiplier norm closure of the polynomials in $\Mult(H^2_d)$.
  Then Drury's von Neumann inequality immediately implies that every commuting
  row contraction admits an $\mathcal{A}_d$-functional calculus. More precisely, we obtain:

  \begin{theorem}
    If $T$ is a commuting row contraction on $\mathcal{H}$,
    then there exists a unital completely contractive homomorphism
    \begin{equation*}
      \mathcal{A}_d \to B(\mathcal{H}), \quad p \mapsto p(T) \quad (p \in \mathbb{C}[z_1,\ldots,z_d]).
    \end{equation*}
  \end{theorem}

  If $d=1$, then the functional calculus above is the disc algebra functional calculus of a contraction.
  Classically, the next step beyond the disc algebra is $H^\infty$.
  Since functions in $H^\infty$ only have boundary values almost everywhere on the unit circle,
  one has to impose an additional condition on the contraction.
  Indeed, the scalar $1$ will not admit a reasonable $H^\infty$-functional calculus.

  Sz.-Nagy and Foias showed that every completely non-unitary contraction (i.e.\ contraction
  without unitary direct summand) admits an $H^\infty$-functional calculus;
  see \cite[Theorem III.2.1]{SFB+10}.
  This is sufficient for essentially all applications, since
  every contraction $T$ decomposes as $T = T_{cnu} \oplus U$ for a completely non-unitary
  contraction $T_{cnu}$ and a unitary $U$, and the unitary part can be analyzed using the spectral theorem.
  The $H^\infty$-functional calculus has been used very successfully
  for instance in the search for invariant subspaces \cite{BCP79,BCP88}.

  We have now seen that for higher dimensions and commuting row contractions,
  the natural replacement for $H^\infty$ is the multiplier
  algebra of the Drury--Arveson space.
  We say that a commuting row contraction $T$ is completely non-unitary
  if it has no spherical unitary summand.
  Once again, one can show that every commuting row contraction decomposes as $T = T_{cnu} \oplus U$
  for a completely non-unitary tuple $T_{cnu}$ and a spherical unitary tuple $U$;
  see \cite[Theorem 4.1]{CD16a}.

  The following theorem, due to Clou\^atre and Davidson \cite{CD16a}, is then a complete generalization of the Sz.-Nagy--Foias $H^\infty$-functional
  calculus to higher dimensions.
  \begin{theorem}
    \label{thm:CD}
    If $T$ is a completely non-unitary commuting row contraction on $\mathcal{H}$, then $T$ admits a $\Mult(H^2_d)$-functional calculus,
    i.e.\ there exists a weak-$*$ continuous unital completely contractive homomorphism
    \begin{equation*}
      \Mult(H^2_d) \to B(\mathcal{H}), \quad p \mapsto p(T) \quad (p \in \mathbb{C}[z_1,\ldots,z_d]).
    \end{equation*}
  \end{theorem}
   
  A different proof, along with a generalization to other reproducing kernel Hilbert spaces,
  was given by Bickel, \mcc\ and the author \cite{BHM17}.

  \begin{remark}
    Every pure row contraction is completely non-unitary, but the converse is false, even if $d=1$.
    For instance, the backwards shift on $\ell^2$ is completely non-unitary, but not pure.

    For pure row contractions, the $\Mult(H^2_d)$-functional calculus can be constructed directly
    with the help of the dilation theorem (Theorem \ref{thm:DA_dilation_pure}).
    Indeed, if $T$ is a pure commuting row contraction on $\mathcal{H}$, then
    by the dilation theorem, there exists an isometry $V: \mathcal{H} \to H^2_d \otimes \mathcal{E}$ with
    \begin{equation*}
      (M_{z_i}^* \otimes I) V = V T_i^* \quad (i=1,\ldots,d).
    \end{equation*}
    Define
    \begin{equation*}
      \Phi: \Mult(H^2_d) \to B(\mathcal{H}), \quad \varphi \mapsto V^* ( M_\varphi \otimes I) V.
    \end{equation*}
    Using weak-$*$ density of the polynomials in $\Mult(H^2_d)$, it is straightforward
    to check that $\Phi$ satisfies the conclusion of Theorem \ref{thm:CD}.
  \end{remark}

 \section{Complete Pick spaces}
 \label{sec:Pick}

 \subsection{Pick's theorem and complete Pick spaces}

 In addition to its special role in multivariable operator theory, the Drury--Arveson space
 also plays a central role in the study of a particular class of reproducing kernel Hilbert spaces,
 called complete Pick spaces. For in-depth information on this topic, see \cite{AM02}.
 The definition of complete Pick spaces is motivated by the following classical interpolation
 theorem from complex analysis due to Pick \cite{Pick15} and Nevanlinna \cite{Nevanlinna19}.

\begin{theorem}
  \label{thm:Pick}
  Let $z_1,\ldots,z_n \in \mathbb{D}$ and $\lambda_1,\ldots,\lambda_n \in \mathbb{C}$.
  There exists $\varphi \in H^\infty$ with
    \begin{equation*}
      \varphi(z_i) = \lambda_i \text{ for } 1 \le i \le n \quad \text{ and } \quad \|\varphi\|_\infty \le 1
    \end{equation*}
  if and only if the matrix
  \begin{equation*}
    \Big[ \frac{1-\lambda_i \overline{\lambda_j}}{1 - z_i \overline{z_j}} \Big]_{i,j=1}^n
  \end{equation*}
  is positive semi-definite.
\end{theorem}

The matrix appearing in Theorem \ref{thm:Pick} is nowadays usually called the \emph{Pick matrix}.

\begin{example}
  \phantomsection
  \label{exa:Pick_small}
  \begin{enumerate}[label=\normalfont{(\alph*)},wide]
    \item If $n=1$, i.e.\ we have a one point interpolation problem $z_1 \mapsto \lambda_1$,
      then the Pick matrix is positive semi-definite if and only if $|\lambda_1| \le 1$.
      This observation also shows that in general, positivity of the Pick matrix implies that $|\lambda_i| \le 1$
      for each $i$, which is obviously a necessary condition for interpolation.
    \item Let $n=2$ and consider the two point interpolation problem $z_i \mapsto \lambda_i$
      for $i=1,2$. Let us assume that $|\lambda_i| \le 1$ for each $i$.
      The Pick matrix takes the form
      \begin{equation*} P =
        \begin{bmatrix}
          \frac{1 - |\lambda_1|^2}{1 - |z_1|^2} & \frac{1 - \lambda_1 \overline{\lambda_2}}{1 - z_1 \overline{z_2}} \\
          \frac{1 - \lambda_2 \overline{\lambda_1}}{1 - z_2 \overline{z_1}} &
          \frac{1 - |\lambda_2|^2}{1 - |z_2|^2}
        \end{bmatrix}.
      \end{equation*}
      Notice that the two diagonal entries are non-negative.

      We distinguish two cases. If $|\lambda_1| = 1$, then the $(1,1)$ entry of $P$ is equal to zero,
      so $P$ is positive if and only if the off-diagonal entries are equal to zero.
      Since $|\lambda_1|=1$, this happens if and only if $\lambda_2 = \lambda_1$.
      Thus, if $|\lambda_1| = 1$, then
      the two point interpolation problem has a solution if and only if $\lambda_2 = \lambda_1$.
      This is nothing else than the maximum modulus principle for holomorphic functions on $\mathbb{D}$:
      If $\varphi: \mathbb{D} \to \overline{\mathbb{D}}$ is holomorphic and if there exists
      $z_1 \in \mathbb{D}$ with $|\varphi(z_1)| = 1$, then $\varphi$ is constant.
      The case $|\lambda_2| = 1$ is analogous.

      Next, suppose that $|\lambda_1|, |\lambda_2| < 1$. By the Hurwitz criterion for positivity,
      the Pick matrix is positive if and only if $\det(P) \ge 0$.
      Using the elementary (but very useful) identity
      \begin{equation*}
        1 - \frac{(1 - |a|^2)(1 - |b|^2)}{|1 - a \overline{b}|^2} = \Big| \frac{a - b}{1 - a \overline{b}} \Big|^2
        \quad (a,b \in \mathbb{D}),
      \end{equation*}
      we find that $\det(P) \ge 0$ if and only if
      \begin{equation*}
        \Big| \frac{\lambda_1 - \lambda_2}{1 - \lambda_1 \overline{\lambda_2}} \Big|
        \le \Big| \frac{z_1 - z_2}{1 - z_1 \overline{z_2}} \Big|.
      \end{equation*}
      These quantities have geometric meaning. Define
      \begin{equation*}
        d_{ph}(z,w) = \Big| \frac{z - w}{1 - z \overline{w}} \Big| \quad (z,w \in \mathbb{D}).
      \end{equation*}
      Then $d_{ph}$ is called the \emph{pseudohyperbolic metric} on $\mathbb{D}$; see \cite[Section I.1]{Garnett07} for background. Therefore, the two point interpolation problem $z_i \mapsto \lambda_i$ for $i=1,2$ has a solution
      if and only if $d_{ph}(\lambda_1,\lambda_2) \le d_{ph}(z_1,z_2)$.
      This statement is usually called the Schwarz--Pick lemma.
  \end{enumerate}
\end{example}

Pick's interpolation theorem (Theorem \ref{thm:Pick}) was proved before Hilbert spaces were abstractly formalized.
However, reproducing kernel Hilbert spaces provide significant insight into Pick's theorem.
The basis for this is the following basic multiplier criterion.
As usual, if $A$ is a Hermitian matrix, we write $A \ge 0$ to mean that $A$ is positive semi-definite.

\begin{proposition}
  \label{prop:mult_crit}
      Let $\mathcal{H}$ be an RKHS on $X$ with kernel $k$.
      Then
      $\varphi \in \Mult(\mathcal{H})$ with $\|\varphi\|_{\Mult(\mathcal{H})} \le 1$ if and only if for all
      finite $F = \{z_1,\ldots,z_n\} \subset X$,
      \begin{equation}
        \label{eqn:Pick_kernel}
        \big[ k(z_i,z_j) (1 - \varphi(z_i) \overline{\varphi(z_j)}) \big]_{i,j=1}^n \ge 0.
      \end{equation}
\end{proposition}

\begin{proof}
  The proof rests on the basic identity $M_\varphi^* k_z = \overline{\varphi(z)} k_z$
  for $\varphi \in \Mult(\mathcal{H})$ and $z \in X$; see Remark \ref{rem:multipliers}.
  Let $\varphi$ be a multiplier of norm at most one. Then $\|M_\varphi^*\| \le 1$
  and so for all $z_1,\ldots,z_n \in X$ and $a_1,\ldots,a_n \in \mathbb{C}$,
  we have
  \begin{equation*}
    \Big\| \sum_{i=1}^n \overline{\varphi(z_i)} a_i k_{z_i} \Big\|^2
    = \Big\| M_\varphi^* \Big( \sum_{i=1}^n a_i k_{z_i} \Big) \Big\|^2
      \le \Big\| \sum_{i=1}^n a_i k_{z_i} \Big\|^2.
  \end{equation*}
  Expanding both sides as a scalar product and rearranging, it follows that
  \begin{equation*}
    \sum_{i,j=1}^n k(z_i,z_j) (1 - \varphi(z_i) \overline{\varphi(z_j)}) a_j \overline{a_i} \ge 0,
  \end{equation*}
  which says that \eqref{eqn:Pick_kernel} holds for $F = \{z_1,\ldots,z_n\}$

  Conversely, suppose that \eqref{eqn:Pick_kernel} holds for all finite subsets of $X$.
  Since linear combinations of kernel functions form a dense subspace of $\mathcal{H}$,
  the computation above shows that there exists a contractive linear operator
  \begin{equation*}
    T: \mathcal{H} \to \mathcal{H}, \quad k_{z} \mapsto \overline{\varphi(z)} k_z.
  \end{equation*}
  As in the proof of Proposition \ref{prop:multiplier_basic}, one shows
  that $T^* = M_\varphi$, so $\varphi \in \Mult(\mathcal{H})$ with $\|\varphi\|_{\Mult(\mathcal{H})} \le 1$.
\end{proof}

Taking $\mathcal{H} = H^2$ and recalling that $\Mult(H^2) = H^\infty$,
it follows that a function $\varphi: \mathbb{D} \to \mathbb{C}$ belongs
to the unit ball of $H^\infty$ if and only if for every finite set $F = \{z_1,\ldots,z_n\} \subset \mathbb{D}$,
\begin{equation*}
  \Big[ \frac{1-\varphi(z_i) \overline{\varphi(z_j)}}{1 - z_i \overline{z_j}} \Big]_{i,j=1}^n \ge 0.
\end{equation*}
This proves necessity in Theorem \ref{thm:Pick}, since the matrix in Theorem \ref{thm:Pick}
corresponds to the special case when $F$ is the set of interpolation nodes.
Conversely, one can use operator theoretic arguments, more specifically commutant lifting,
to prove sufficiency in Theorem \ref{thm:Pick}. This approach to the Nevanlinna--Pick interpolation problem was pioneered by
Sarason \cite{Sarason67}.
We will provide a proof of sufficiency following Theorem \ref{thm:BTV} below.

Viewing Theorem \ref{thm:Pick} as a statement about the reproducing kernel Hilbert space $H^2$
raises an obvious question: For which reproducing kernel Hilbert spaces is Pick's theorem true?

\begin{definition}
  An RKHS $\mathcal{H}$ with kernel $k$ is said to be a \emph{Pick space}
  if whenever $z_1,\ldots,z_n \in X$ and $\lambda_1,\ldots,\lambda_n \in \mathbb{C}$ with
  \begin{equation*}
    [k(z_i,z_j) ( 1 -\lambda_i \overline{\lambda_j} ) ]_{i,j=1}^n \ge 0,
  \end{equation*}
  then there exists $\varphi \in \Mult(\mathcal{H})$ with
  \begin{equation*}
    \varphi(z_i) = \lambda_i \text{ for } 1 \le i \le n \quad \text{ and } \quad \|\varphi\|_{\Mult(\mathcal{H})} \le 1.
  \end{equation*}
\end{definition}

\begin{example}
  \label{exa:Bergman_not_Pick}
  The Hardy space $H^2$ is a Pick space by Pick's interpolation theorem \ref{thm:Pick}.

  The Bergman space
    \begin{equation*}
      L^2_a = L^2_a(\mathbb{D}) = \Big\{ f \in \mathcal{O}(\mathbb{D}): \|f\|^2_{L^2_a} = \int_{\mathbb{D}} |f|^2 d A < \infty \Big\},
    \end{equation*}
    where $A$ is the normalized area measure on $\mathbb{D}$, is not a Pick space.
  To see this, observe that $\Mult(L^2_a) = H^\infty$, and that the multiplier norm
  is the supremum norm. Moreover, the reproducing kernel of $L^2_a$ is given by
  \begin{equation*}
    k(z,w) = \frac{1}{(1 - z \overline{w})^2} \quad (z,w \in \mathbb{D});
  \end{equation*}
  see, for instance, \cite[Section 1.2]{DS04}.
  Thus, the reproducing kernel of $L^2_a$ is the square of the reproducing kernel of $H^2$,
  but the multiplier algebras agree.
  However, it is a simple matter to find points in $\mathbb{D}$ whose Pick matrix
  with respect to the kernel of $L^2_a$ is positive, but whose Pick matrix with respect to the kernel of $H^2$ is not positive. For instance,
  let $z_1 = 0, z_2 = r \in (0,1)$ and $\lambda_1 = 0, \lambda_2 = t \in (0,1)$.
  Then the Pick matrix with respect to $L^2_a$ is
  \begin{equation*}
    \begin{bmatrix}
      1 & 1 \\
      1 & \frac{1 - t^2}{(1 - r^2)^2}
    \end{bmatrix},
  \end{equation*}
  which is positive if and only if
  \begin{equation*}
    t^2 \le r^2 (2 - r^2).
  \end{equation*}
  Choosing $t$ so that equality holds, we see that $t > r$,
  hence the Pick matrix with respect to $H^2$ is not positive.
  (Alternatively, the two point interpolation problem cannot have a solution by the Schwarz lemma.)
\end{example}

\begin{remark}
  One can show that a Pick space is uniquely determined by its multiplier algebra; see for instance \cite[Corollary 3.2]{Hartz17a} for a precise statement.
  Thus, there exists at most one Pick space with a given multiplier algebra. This generalizes the argument
  for the Bergman space in Example \ref{exa:Bergman_not_Pick}.
\end{remark}

It turns out that one obtains a significantly cleaner theory by not only considering interpolation
with scalar targets, but also with matrix targets.
This can be regarded as part of a more general principle in functional analysis,
namely that demanding that certain properties hold at all matrix levels often has powerful consequences;
see for instance \cite{Arveson69,Paulsen02}.

\begin{definition}
  An RKHS $\mathcal{H}$ with kernel $k$ is said to be a \emph{complete Pick space}
  if whenever $r \in \mathbb{N}, r \ge 1$ and $z_1,\ldots,z_n \in X$ and $\Lambda_1,\ldots,\Lambda_N \in M_r(\mathbb{C})$ with
  \begin{equation*}
    [k(z_i,z_j) (I - \Lambda_i \Lambda_j^*)]_{i,j=1}^n \ge 0,
  \end{equation*}
then there exists $\Phi \in M_r(\Mult(\mathcal{H}))$ with
  \begin{equation*}
    \Phi(z_i) = \lambda_i \text{ for } 1 \le i \le n \quad \text{ and } \quad \|\Phi\|_{M_r(\Mult(\mathcal{H}))} \le 1.
  \end{equation*}
\end{definition}

Instead of saying that $\mathcal{H}$ is a (complete) Pick space, we will also say that the reproducing
kernel $k$ is a (complete) Pick kernel.

It may not be immediately clear how useful it is to turn Pick's interpolation theorem into a definition.
However, we will see that the complete Pick property of a space $\mathcal{H}$ has very powerful consequences
for the function theory and operator theory associated with $\mathcal{H}$.
Moreover, the complete Pick property is satisfied by many familiar RKHS.

Before turning to examples of complete Pick spaces and to consequences of the complete
Pick property, we consider a reformulation.
Let $\mathcal{H}$ be an RKHS on $X$ with kernel $k$. For $Y \subset X$, let
\begin{equation*}
  \mathcal{H} \big|_Y = \{ f \big|_Y : f \in \mathcal{H}\}
\end{equation*}
with the quotient norm $\|g\|_{\mathcal{H} |_Y} = \inf \{\|f\|_{\mathcal{H}}: f \big|_Y = g\}$.
This is an RKHS on $Y$ with kernel $k \big|_{Y \times Y}$; see for instance \cite[Corollary 5.8]{PR16}.
Thus, by definition,
\begin{equation*}
  R: \mathcal{H} \to \mathcal{H} \big|_Y, \quad f \mapsto f \big|_Y,
\end{equation*}
is a quotient mapping, meaning that it maps the open unit ball onto the open unit ball.
Note that $\ker(R)^\bot = \bigvee\{k_y: y \in Y\} =: \mathcal{H}_Y$, so $R$ induces a unitary operator
between $\mathcal{H}_Y$ and $\mathcal{H} \big|_Y$.

On the level of multipliers, we always obtain a complete contraction
\begin{equation*}
  \Mult(\mathcal{H}) \to \Mult(\mathcal{H} \big|_Y), \quad \varphi \mapsto \varphi \big|_Y.
\end{equation*}

The Pick property precisely says that one also obtains a quotient mapping on the level of multipliers.

\begin{proposition}
  \label{prop:Pick_reform}
  Let $\mathcal{H}$ be an RKHS on $X$ with kernel $k$.
  The following assertions are equivalent:
  \begin{enumerate}[label=\normalfont{(\roman*)}]
    \item $\mathcal{H}$ is a (complete) Pick space;
    \item for each finite set $F \subset X$, the map $\Mult(\mathcal{H}) \to \Mult(\mathcal{H} \big|_F),
      \varphi \mapsto \varphi \big|_F$, is a (complete) quotient map;
    \item for each finite $F \subset Y$ and all $\psi \in \Mult(\mathcal{H} \big|_F)$
      there exists $\varphi \in \Mult(\mathcal{H})$
      with $\varphi \big|_F = \psi$ and $\|\varphi\|_{\Mult(\mathcal{H})} = \|\psi\|_{\Mult(\mathcal{H} \big|_F)}$
      (respectively the same property for all $r \in \mathbb{N}, r \ge 1$ and all $\psi \in M_r(\Mult(\mathcal{H} \big|_F)$).
  \end{enumerate}
\end{proposition}

\begin{proof}
  (i) $\Leftrightarrow$ (iii)
  Let $F = \{z_1,\ldots,z_n\}$ be a finite subset of $X$
  and let
  \begin{equation*}
    \psi: F \to \mathbb{C}, \quad z_i \mapsto \lambda_i \quad (i=1,\ldots,n).
  \end{equation*}
  Then Proposition \ref{prop:mult_crit}, applied to the space $\mathcal{H} |_F$, shows
  that $\|\psi\|_{\Mult(\mathcal{H} \big|_F)} \le 1$ if and only if the Pick matrix
  for the interpolation problem $z_i \mapsto \lambda_i$ is positive.
  Thus, $\mathcal{H}$ is a Pick space if and only if (iii) holds, and a similar argument
  applies in the matrix valued setting.

  (iii) $\Rightarrow$ (ii) is trivial.

  (ii) $\Rightarrow$ (iii)
  By assumption, the restriction map $\Mult(\mathcal{H}) \to \Mult(\mathcal{H} \big|_F)$ maps the open
  unit ball onto the open unit ball. We have to show that the image of the closed unit ball
  contains the closed unit ball. This follows from a weak-$*$ compactness argument.

  Recall from the discussion preceding Proposition \ref{prop:weak-star}
  that $\Mult(\mathcal{H})$ and $\Mult(\mathcal{H} \big|_F)$ carry natural weak-$*$ topologies.
  With respect to these, the restriction map is weak-$*$--weak-$*$ continuous,
  since modulo the identification $\mathcal{H} \big|_F \cong \mathcal{H}_F$ explained
  before the proposition, the restriction map is given by $M_\varphi \mapsto P_{\mathcal{H}_F} M_\varphi \big|_{\mathcal{H}_F}$.
  By Alaoglu's theorem, the closed unit ball of $\Mult(\mathcal{H})$ is weak-$*$ compact,
  hence so is its image under the restriction map.
  In particular, the image is norm closed and hence contains the closed unit ball.
\end{proof}

\subsection{Characterizing complete Pick spaces}
Complete Pick spaces turn out to have a clean characterization, thanks to a theorem of McCullough
\cite{McCullough92,McCullough94} and Quiggin \cite{Quiggin93}.
We will use a version of the result due to Agler and \mcc\ \cite{AM00}.
Let us say that a function $F: X \times X \to \mathbb{C}$ is positive if $[F(x_i,x_j)]_{i,j=1}^n$ is positive semi-definite for any $x_1,\ldots,x_n \in X$. We also write $F \ge 0$.

  \begin{theorem}
    \label{thm:mc_q}
    Let $\mathcal{H}$ be an RKHS on $X$ with kernel $k$. Assume
    that $k(x,y) \neq 0$ for all $x,y \in X$.
    Then $\mathcal{H}$ is a complete Pick space if and only if for some
    $z \in X$,
    \begin{equation*}
      F_z: X \times X \to \mathbb{C}, \quad (x,y) \mapsto 1 - \frac{k(x,z) k(z,y)}{k(z,z) k(x,y)},
    \end{equation*}
    is positive. In this case, $F_z$ is positive for all $z \in X$.
  \end{theorem}
  There are now several proofs of this result, see for instance \cite[Chapter 7]{AM02}.
  A simple proof of necessity was recently obtained by Knese \cite{Knese19a}.
  Just as in the classical case of $H^2$, it is possible to deduce sufficiency
  from a suitable commutant lifting theorem; see \cite[Section 5]{BTV01} and \cite{AT02}.
  In this context, the non-commutative approach mentioned in Subsections \ref{ss:nc} and \ref{ss:nc_mult}
  is also useful; see \cite[Section 4]{DP98}, \cite{AP00}, \cite[Section 5]{DL10} and \cite{DH11}.
  We will shortly discuss sufficiency in a bit more detail, but let us first consider some examples.

  Theorem \ref{thm:mc_q} becomes even easier to state if one adds a normalization hypothesis.
  A kernel $k:X \times X \to \mathbb{C}$ is said to be \emph{normalized} at $x_0 \in X$ if $k(x,x_0) = 1$ for all $x \in X$. In this case, we say that $k$ (or $\mathcal{H}$) is \emph{normalized}.
  Many kernels of interest are normalized at a point. For instance, the Drury--Arveson kernel
  is normalized at the origin. In the abstract setting, we think of the normalization point
  as playing the role of the origin.
  Any non-vanishing kernel can be normalized by considering
  \begin{equation*}
    \widetilde{k}(x,y) = \frac{k(x,y) k(x_0,x_0)}{k(x,x_0) k(y,x_0)}
  \end{equation*}
  for some $x_0 \in X$.
  This normalization neither changes the multiplier algebra nor positivity of Pick matrices,
  and so it can usually be assumed without loss of generality.
  See also \cite[Section 2.6]{AM02} for more discussion.

  If we set $z$ to be the normalization point in Theorem \ref{thm:mc_q}, then we obtain the following
  result.
  \begin{corollary}
    \label{cor:mc_q_normalized}
  If $k$ is normalized, then $k$ is a complete Pick kernel
  if and only if $1 - \frac{1}{k}$ is positive.
  \end{corollary}

  The last result can be used to quickly give examples of complete Pick spaces.

  \begin{corollary}
    \label{cor:DA_complete_Pick}
    The Drury--Arveson space is a complete Pick space.
  \end{corollary}

  \begin{proof}
    The function $1 - \frac{1}{k(z,w)} = \langle z,w \rangle$ is positive.
  \end{proof}

  For $a >0$, let $\mathcal{H}_a$ be the RKHS on $\mathbb{D}$ with kernel $\frac{1}{(1 -z  \overline{w})^a}$.
  We already encountered these spaces in Subsection \ref{ss:scale}.
  \begin{corollary}
    \label{cor:weighted_dirichlet}
    $\mathcal{H}_a$ is a complete Pick space if and only if $0 <a \le 1$.
  \end{corollary}

  \begin{proof}
    We have
    \begin{equation*}
      1 - (1 - z \overline{w})^a = \sum_{n=1}^\infty (-1)^{n-1} \binom{a}{n} (z \overline{w})^n,
    \end{equation*}
    where
    $\binom{a}{n} = \frac{a (a-1) \ldots (a-n+1)}{n!}$.
    A function of this type is positive if and only if all coefficients in the sum are non-negative;
    see for instance the proof of Theorem 7.33 in \cite{AM02} or \cite[Corollary 6.3]{Hartz17a}.

    If $0 < a \le 1$, then $(-1)^{n-1} \binom{a}{n} \ge 0$ for all $n$,
    so $1- \frac{1}{k}$ is positive.
    If $a > 1$, then $(-1)^{2-1} \binom{a}{2} < 0$, hence $1-\frac{1}{k}$ is not positive.
  \end{proof}

  The spaces $\mathcal{H}_a$ for $a \in (0,1]$ are weighted Dirichlet spaces.
  Recall that the classical Dirichlet space is
  \begin{equation*}
    \mathcal{D} = \Big\{ f \in \mathcal{O}(\mathbb{D}): \int_{\mathbb{D}} |f'|^2 d A < \infty \Big\},
  \end{equation*}
  where $dA$ denotes integration with respect to normalized area measure.
  We equip $\mathcal{D}$ with the norm
  \begin{equation*}
    \|f\|_{\mathcal{D}}^2 = \|f\|_{H^2}^2 + \int_{\mathbb{D}} |f'|^2 d A.
  \end{equation*}
  Once again, more information on the Dirichlet space can be found in \cite{ARS+19,EKM+14}.
  The following important result of Agler \cite{Agler88a} in fact predates Theorem \ref{thm:mc_q}.
  It initiated the study of complete Pick spaces.
  
  \begin{theorem}
    The classical Dirichlet space is a complete Pick space.
  \end{theorem}

  \begin{proof}
    The reproducing kernel of $\mathcal{D}$ is given by
    \begin{equation*}
      k(z,w) = \frac{1}{z \overline{w}} \log \Big( \frac{1}{1 - z \overline{w}} \Big)
      = \sum_{n=0}^\infty \frac{1}{n+1} (z \overline{w})^n,
    \end{equation*}
    see for instance \cite[Theorem 1.2.3]{EKM+14}.
    This kernel is normalized at $0$.
    The coefficients of $1 - 1/k$ are more difficult to compute than those
    in Corollary \ref{cor:weighted_dirichlet}. Instead, the standard
    approach is to use a Lemma of Kaluza (see \cite[Lemma 7.38]{AM02}),
    which says that log-convexity of the sequence $\frac{1}{n+1}$ implies that
    all coefficients of $1-1/k$ are non-negative; whence Corollary \ref{cor:mc_q_normalized} applies.
    (A sequence $(a_n)$ of positive numbers is log-convex if $a_n^2 \le a_{n-1} a_{n+1}$ for all $n \ge 1$.)
    For an explicit computation of the coefficients of $1-1/k$, see for instance \cite[Section 5]{Hartz22}.
  \end{proof}

  \begin{remark}
    The complete Pick property is an isometric property and generally not stable under passing
    to an equivalent norm. For instance, a frequently used equivalent norm on the Dirichlet space is given by
    \begin{equation*}
      \|f\|_{\widetilde{\mathcal{D}}}^2 = |f(0)|^2 + \int_{\mathbb{D}} |f'|^2 \,d A.
    \end{equation*}
    The reproducing kernel with respect to this norm is
    \begin{equation*}
      \widetilde{k}(z,w) = 1 + \log\Big( \frac{1}{1 - z \overline{w}} \Big)
      = 1 + \sum_{n=1}^\infty \frac{1}{n} (z \overline{w})^n,
    \end{equation*}
    and one computes that the coefficient of $(z \overline{w})^2$ in $1-1/\widetilde{k}$
    is $-\frac{1}{2} < 0$, so the Dirichlet space is not a complete Pick space in the norm $\|\cdot\|_{\widetilde{\mathcal{D}}}$.
  \end{remark}

  Let us briefly discuss sufficiency in Theorem \ref{thm:mc_q}.
  For simplicity, we only consider the case of the Drury--Arveson space;
  in other words, Corollary \ref{cor:DA_complete_Pick}.

  The approach to proving the complete Pick property of the Drury--Arveson space we discuss here
  is to first establish a realization formula for multipliers; see for instance \cite[Chapter 8]{AM02}
  and \cite[Chapter 2]{AMY20} for background on this topic.
  The following theorem was obtained by Ball, Trent and Vinnikov (see Theorem 2.1 and Theorem 4.1 in \cite{BTV01}),
  and by Eschmeier and Putinar (see Proposition 1.2 in \cite{EP02}).

  \begin{theorem}
    \label{thm:BTV}
    Let $Y \subset \mathbb{B}_d$ and let $\varphi: Y \to \mathbb{C}$.
    The following statements are equivalent:
    \begin{enumerate}[label=\normalfont{(\roman*)}]
      \item $\|\varphi\|_{\Mult(H^2_d |_Y)} \le 1$.
      \item There exists a Hilbert space $\mathcal{E}$ and a unitary
        \begin{equation*}
          U =
          \begin{bmatrix}
            A & B \\ C & D
          \end{bmatrix}:
          \mathcal{E} \oplus \mathbb{C} \to \mathcal{E}^d \oplus \mathbb{C}
        \end{equation*}
        with
        \begin{equation*}
          \varphi(z) = D + C (I_{\mathcal{E}} - Z(z) A)^{-1} Z(z) B,
        \end{equation*}
        where $Z(z) =
        \begin{bmatrix}
          z_1 & \ldots & z_d
        \end{bmatrix}: \mathcal{E}^d \to \mathcal{E}$.
    \end{enumerate}
  \end{theorem}

  \begin{proof}
    (i) $\Rightarrow$ (ii)
    The argument is known as a lurking isometry argument. It works as follows.
    Suppose that $\|\varphi\|_{\Mult(H^2_d |_Y)} \le 1$.
    Proposition \ref{prop:mult_crit} shows that
    \begin{equation}
      \label{eqn:mult_pos}
      (z,w) \mapsto \frac{1 - \varphi(z) \overline{\varphi(w)}}{1 - \langle z,w \rangle }
    \end{equation}
    is positive on $Y \times Y$. Therefore, this function admits a Kolmogorov decomposition,
    meaning that there is a Hilbert space $\mathcal{E}$ and a map $h: Y \to \mathcal{E}$ so that
    \begin{equation*}
      \frac{1 - \varphi(z) \overline{\varphi(w)}}{1 - \langle z,w \rangle } = \langle h(w),h(z) \rangle_{\mathcal{E}}
      \quad \text{ for all } z,w \in Y.
    \end{equation*}
    (One can take $\mathcal{E}$ to be the RKHS on $Y$ with reproducing kernel \eqref{eqn:mult_pos}, and $h(z)$ to be the kernel at $z$ in $\mathcal{E}$;
    see \cite[Theorem 2.14]{PR16} or the first proof of Theorem 2.53 in \cite{AM02}.)
    Rearranging and collecting positive terms, we see that
    \begin{equation*}
      \langle z,w \rangle \langle h(w), h(z) \rangle  + 1 = \langle h(w), h(z) \rangle  + \varphi(z) \overline{\varphi(w)}.
    \end{equation*}
    Now comes the key observation. The last equation says that
    \begin{equation*}
      \left \langle
        \begin{pmatrix}
          \overline{w_1} h(w) \\ \vdots \\ \overline{w_d} h(w) \\ 1
        \end{pmatrix},
        \begin{pmatrix}
          \overline{z_1} h(z) \\ \vdots \\ \overline{z_d} z(w) \\ 1
        \end{pmatrix}
      \right \rangle_{\mathcal{E}^d \oplus \mathbb{C}}
      = \left \langle
      \begin{pmatrix}
        h(w) \\ \overline{\varphi(w)}
      \end{pmatrix},
      \begin{pmatrix}
        h(z) \\ \overline{\varphi(z)}
      \end{pmatrix}
    \right \rangle_{\mathcal{E} \oplus \mathbb{C}}
    \end{equation*}
    for all $z,w \in Y$.
    Therefore, we may define a linear isometry $V$ by
    \begin{equation*}
      V:
      \begin{pmatrix}
        \overline{z_1} h(z) \\ \vdots \\ \overline{z_d} h(z) \\ 1
      \end{pmatrix} \mapsto
      \begin{pmatrix}
        h(z) \\ \overline{\varphi(z)}
      \end{pmatrix},
    \end{equation*}
    initially mapping the closed linear span of vectors appearing on the left onto the closed linear span
    of vectors appearing on the right. By enlarging $\mathcal{E}$ if necessary, we may extend $V$ to a unitary
    $\mathcal{E}^d \oplus \mathbb{C} \to \mathcal{E} \oplus \mathbb{C}$. Let $U = V^*$.
    Decomposing
    \begin{equation*}
      U =
      \begin{bmatrix}
        A & B \\ C & D
      \end{bmatrix}: \mathcal{E} \oplus \mathbb{C} \to \mathcal{E}^d \oplus \mathbb{C},
    \end{equation*}
    we find from the definition of $V$ that
    \begin{align*}
      A^* Z(z)^* h(z) + C^* &= h(z) \quad \text{ and } \\
      B^* Z(z)^* h(z) + D^* &= \overline{\varphi(z)}.
    \end{align*}
    Solving the first equation for $h(z)$, we obtain
    \begin{equation*}
      h(z) = (I_{\mathcal{E}} - A^* Z(z)^*)^{-1} C^*;
    \end{equation*}
    observe that $\|A^*\| \le 1$ since $U$ is unitary and hence $\|A^* Z(z)^*\| < 1$ for all $z \in \mathbb{B}_d$,
    so $I - A^* Z(z)^*$ is indeed invertible. Substituting the result into the second equation, we conclude that
    \begin{equation*}
    \overline{\varphi(z)} = D^* + B^* Z(z)^* (I_{\mathcal{E}} - A^* Z(z)^*)^{-1} C^*,
    \end{equation*}
    which gives (ii) after taking adjoints.

    (ii) $\Rightarrow$ (i) Suppose that (ii) holds and define
    \begin{equation*}
      h(z) = (I_{\mathcal{E}} - A^* Z(z)^*)^{-1} C^* \in \mathcal{E}.
    \end{equation*}
    Reversing the steps in the proof of (i) $\Rightarrow$ (ii), it follows that
    \begin{equation*}
      \frac{1 - \varphi(z) \overline{\varphi(w)}}{1 - \langle z,w \rangle } = \langle h(w), h(z) \rangle_{\mathcal{E}} \quad \text{ for all } z,w \in Y,
    \end{equation*}
    hence $\varphi$ is a multiplier of $\Mult(H^2_d \big|_Y)$ of norm at most one by Proposition \ref{prop:mult_crit}.
  \end{proof}

  Using Theorem \ref{thm:BTV}, it is a simple matter to show that the Drury--Arveson space
  satisfies the Pick property.
  In fact, Theorem \ref{thm:BTV} holds in the vector-valued setting,
  which gives the complete Pick property of $H^2_d$.
  For simplicity, we restrict our attention to the scalar case.

  \begin{proof}[Proof of Pick property of $H^2_d$]
    Let $Y \subset \mathbb{B}_d$ and suppose that $\varphi: Y \to \mathbb{C}$ satisfies
    $\|\varphi\|_{\Mult(H^2_d|_Y)} \le 1$.
    Applying (i) $\Rightarrow$ (ii) of Theorem \ref{thm:BTV}, we obtain a realization
  \begin{equation*}
    \varphi(z) = D + C(I_{\mathcal{E}} - Z(z) A)^{-1} Z(z) B \quad (z \in Y).
  \end{equation*}
  The same formula defines an extension $\Phi$ of $\varphi$ to $\mathbb{B}_d$;
  note that $\|Z(z) A\| < 1$ for $z \in \mathbb{B}_d$, so the inverse exists.
  By (ii) $\Rightarrow$ (i) of Theorem \ref{thm:BTV},
  $\|\Phi\|_{\Mult(H^2_d)} \le 1$. Thus, $H^2_d$ is a Pick space; see Proposition \ref{prop:Pick_reform}.
  \end{proof}

\begin{remark}
  \label{rem:arbitrary_extension}
  The proof above shows that we may extend multipliers defined on arbitrary subsets of $\mathbb{B}_d$,
  not only finite ones. This can also be directly seen from the definition of the (complete) Pick property
  by means of a weak-$*$ compactness argument.
\end{remark}

\subsection{Universality of the Drury--Arveson space}

We saw in the last subsection that the Drury--Arveson space is a complete Pick space.
Remarkably, it is much more than an example. If one is willing to include the case $d=\infty$,
then the Drury--Arveson space is a universal complete Pick space.

To explain this, recall from Corollary \ref{cor:mc_q_normalized} that a normalized kernel
$k$ is a complete Pick kernel if and only if $1-\frac{1}{k}$ is positive.
(We again only consider normalized kernels for the sake of simplicity.)
Applying the Kolmogorov factorization to the function $1-\frac{1}{k}$, we obtain
a Hilbert space $\mathcal{E}$ and $b: X \to \mathcal{E}$ satisfying $\|b(x)\| < 1$ for all $x \in X$ and
\begin{equation*}
  1 - \frac{1}{k(x,y)} = \langle b(x), b(y) \rangle.
\end{equation*}
(This is the same procedure as in the proof of Theorem \ref{thm:BTV}, with the exception that the roles of $x$ and $y$ on the right-hand side are reversed.
To achieve this, one may pass from the Hilbert space $\mathcal{E}$ to the conjugate Hilbert space $\overline{\mathcal{E}}$; see also the first proof of Theorem 2.53 in \cite{AM02}.)
Conversely, if $1-\frac{1}{k}$ has the form as above, then it is positive.
Thus, we obtain the following theorem, due to Agler and \mcc\ \cite{AM00}.

\begin{theorem}
  \label{thm:AM}
  An RKHS $\mathcal{H}$ with normalized kernel $k$ is a complete Pick space if and only if there
  exists a Hilbert space $\mathcal{E}$ and $b: X \to \mathcal{E}$ such that $\|b(x)\| < 1$ for all $x \in X$ and
  \begin{equation*}
    k(x,y) = \frac{1}{1 - \langle b(x), b(y) \rangle }.
  \end{equation*}
\end{theorem}

Before continuing, we make two more simplifying assumptions:
\begin{enumerate}[label=\normalfont{(\arabic*)}]
  \item $\mathcal{H}$ separates the points of $X$, and
  \item $\mathcal{H}$ is separable.
\end{enumerate}

The first condition implies that $k(\cdot,y_1) \neq k(\cdot,y_2)$ whenever $y_1,y_2 \in X$
with $y_1 \neq y_2$, hence $b$ is injective.
The second condition implies that the space $\mathcal{E}$ can be taken to be separable
as well, so we may assume that $\mathcal{E} = \mathbb{C}^d$ or $\mathcal{E} = \ell^2$.
(This follows from the fact that an RKHS is separable if and only if it admits a countable
set of uniqueness, and the RKHS with kernel $\langle b(x), b(y) \rangle$ is contained in $\mathcal{H}$.)

Notice that if $\mathcal{E} = \mathbb{C}^d$, then Theorem \ref{thm:AM} expresses the kernel
$k$ as a composition of the Drury--Arveson kernel $\frac{1}{1 - \langle z,w \rangle }$ and the embedding $b$.
If $d=\infty$, we need to extend our definition of the Drury--Arveson space.

\begin{definition}
  \label{def:DA_infinite}
  Let $\mathbb{B}_\infty = \{z \in \ell^2: \|z\|_2 < \infty\}$. The Drury--Arveson space $H^2_\infty$
  is the RKHS on $\mathbb{B}_\infty$ with reproducing kernel 
  \begin{equation*}
    K(z,w) = \frac{1}{1 - \langle z,w \rangle_{\ell^2}}.
  \end{equation*}
\end{definition}

Corollary \ref{cor:mc_q_normalized} shows that $H^2_\infty$ is also a complete Pick space.
On the level of function spaces, Theorem \ref{thm:AM}
means that there exist $d \in \mathbb{N} \cup \{\infty\}$ and an isometry
\begin{equation*}
  V: \mathcal{H} \to H^2_d, \quad k(\cdot,y) \mapsto \frac{1}{1 - \langle \cdot, b(y) \rangle }.
\end{equation*}
Thus, we obtain an embedding of $\mathcal{H}$ into $H^2_d$ that sends kernel functions to kernel functions.
Taking the adjoint of $V$, we obtain a composition operator that is co-isometric
(which is the same as a quotient map in the Hilbert space setting).
Thanks to the complete Pick property, this carries over to multiplier algebras,
which leads to the following result, again due to Agler and \mcc\ \cite{AM00}.
It shows the special role that the Drury--Arveson space plays in the theory of complete Pick spaces.

\begin{theorem}[Universality of the Drury--Arveson space]
  \label{thm:AM_2}
  If $\mathcal{H}$ is a normalized complete Pick space on $X$, then there exist $1 \le d \le \infty$
  and a map $b: X \to \mathbb{B}_d$ such that
  \begin{equation*}
    H^2_d \to \mathcal{H}, \quad f \mapsto f \circ b,
  \end{equation*}
  is a co-isometry, and
  \begin{equation*}
    \Mult(H^2_d) \to \Mult(\mathcal{H}), \quad \varphi \mapsto \varphi \circ b,
  \end{equation*}
    is a complete quotient map.
    In fact, if $\varphi \in \Mult(\mathcal{H})$, then
  there exists $\Phi \in \Mult(H^2_d)$ with $\varphi = \Phi \circ b$
  and $\|\Phi\|_{\Mult(H^2_d)} = \|\varphi\|_{\Mult(\mathcal{H})}$.
\end{theorem}

\begin{proof}
  The first statement was already explained. The second statement follows from the complete Pick property
  of $H^2_d$. Indeed, let $Y = b(X)$. If $\varphi \in \Mult(\mathcal{H})$, define
  $\psi: Y \to \mathbb{C}$ by $\psi(b(x)) = \varphi(x)$.
  The relation $k(x,y) = \frac{1}{1 - \langle b(x), b(y) \rangle }$ implies
  that $\psi \in \Mult(H^2_d \big|_Y)$ with $\|\psi\|_{\Mult(H^2_d|_Y)} = \|\varphi\|_{\Mult(\mathcal{H})}$,
  for instance by Proposition \ref{prop:mult_crit}.
  By the complete Pick property of $H^2_d$, we may extend $\psi$ to a multiplier of $H^2_d$
  of the same norm, see also Remark \ref{rem:arbitrary_extension}.
  The same proof works for matrix valued multipliers.
\end{proof}

The Agler-\mcc\ universality theorem is sometimes summarized as saying that
``every complete Pick space is a quotient of the Drury--Arveson space''.

\begin{remark}
  For many spaces of interest, including the classical Dirichlet space
  and the weighted Dirichlet spaces $\mathcal{H}_a$ for $a \in (0,1)$ (see Corollary \ref{cor:weighted_dirichlet}),
  one has to take $d=\infty$ in Theorem \ref{thm:AM_2}.
  This can be seen by determining the rank of the kernel $1-\frac{1}{k}$; see Proposition 11.8 and Corollary 11.9 in \cite{Hartz17a} and also \cite{Rochberg16} for a geometric approach. In the case of the classical Dirichlet space,
  $d=\infty$ is necessary even if one weakens the conclusions in Theorem \ref{thm:AM_2} in certain ways;
  see \cite{Hartz22} for a precise statement.
\end{remark}

It is possible to push this line of reasoning even further, which was done by Davidson, Ramsey and Shalit
\cite{DRS15}.
First, we note that Theorem \ref{thm:AM_2} can be further reformulated to say that defining $V = b(X)$, we obtain a unitary
\begin{equation*}
  H^2_d \big|_V \to \mathcal{H}, \quad f \mapsto f \circ b.
\end{equation*}

Which subsets $V$ of $\mathbb{B}_d$ are relevant?

\begin{definition}
    If $S \subset H^2_d$, let  
    \begin{equation*}
      V(S) = \{z \in \mathbb{B}_d: f(z) = 0 \text{ for all } f \in S\}.
    \end{equation*}
    If $W \subset \mathbb{B}_d$, let 
    \begin{equation*}
      I(W) = \{ f \in H^2_d: f \big|_W = 0 \}.
    \end{equation*}
    We say that $V \subset \mathbb{B}_d$ is a \emph{variety} if $V = V(S)$ for some $S \subset H^2_d$.
\end{definition}

Note that if $d=1$, then the varieties are precisely the Blaschke sequences, along with the entire unit disc.
If $1 \le d < \infty$, then every algebraic variety is a variety in the sense above, as $H^2_d$ contains
all polynomials. Conversely, every variety in the sense above is an analytic variety in $\mathbb{B}_d$, as $H^2_d$ consists
of holomorphic functions.

If $W \subset \mathbb{B}_d$ is arbitrary, let $V = V(I(W)) \supset W$.
Clearly, $V$ is the smallest variety containing $W$.
We think of $V$ as an analogue of the Zariski closure in algebraic geometry.
It is immediate from the definitions that
\begin{equation*}
  H^2_d \big|_V \to H^2_d \big|_W, \quad f \mapsto f|_W,
\end{equation*}
is unitary. In particular, every function $H^2_d \big|_W$ has a unique extension to a function in $H^2_d \big|_V$.

Thus, we obtain the following refined version of Theorem \ref{thm:AM_2}, due to Davidson, Ramsey and Shalit \cite{DRS11}.

  \begin{theorem}
    If $\mathcal{H}$ is a normalized complete Pick space on $X$, then
    there exists a variety $V \subset \mathbb{B}_d$ for some $1 \le d \le \infty$
    and a map $b: X \to V$ such that
    \begin{equation*}
      C_b: H^2_d \big|_V \to \mathcal{H}, \quad f \mapsto f \circ b,
    \end{equation*}
    is unitary.
  In this case,
  \begin{equation*}
    \Mult(H^2_d \big|_V) \to \mathcal{H}, \quad \varphi \mapsto \varphi \circ b,
  \end{equation*}
  is a completely isometric isomorphism.
  \end{theorem}
  Note that if $\varphi \in \Mult(H^2_d \big|_V)$, then $C_b M_\varphi C_{b}^* = M_{\varphi \circ b}$,
  so the completely isometric isomorphism if given by conjugation with the unitary $C_b$.

  By the complete Pick property of $H^2_d$,
  \begin{equation*}
    \mathcal{M}_V : =\Mult(H^2_d \big|_V) = \Mult(H^2_d) \big|_V.
  \end{equation*}

  This raises the following natural question.

  \begin{question}
    What is the relationship between the operator algebra structure of $\mathcal{M}_V$
    and the geometry of $V$?
  \end{question}

  The study of this question was initiated by Davidson, Ramsey and Shalit in \cite{DRS11,DRS15}.
  At the Focus Program, Orr Shalit provided a survey of results in this area.
  For more details, the reader is referred to the survey article \cite{SS14}.
  Other references are \cite{DHS15,Hartz12,Hartz17a,HL18,KMS13}.

\section{Selected topics}

This last section contains a few selected topics regarding the Drury--Arveson space
and more generally complete Pick spaces.
The selection was made based on which other topics were covered at the Focus Program,
but it is admittedly also heavily influenced by the author's taste.
For other selections and perspectives, we once again refer the reader to \cite{Shalit13} and
\cite{FX19}.

\subsection{Maximal ideal space and corona theorem}
\label{ss:corona}

We assume that $d< \infty$.
The multiplier algebra of the Drury--Arveson space $H^2_d$ is a unital commutative Banach algebra.
Hence Gelfand theory strongly suggests that one should study its maximal ideal space
\begin{equation*}
  \mathcal{M} (\Mult(H^2_d)) = \{ \chi: \Mult(H^2_d) \to \mathbb{C}: \chi \text{ is linear, multiplicative, $\neq 0$}\},
\end{equation*}
which is a compact Hausdorff space in the weak-$*$ topology.
Naturally, we obtain a topological embedding 
\begin{equation*}
  \mathbb{B}_d \hookrightarrow \mathcal{M}(\Mult(H^2_d))
\end{equation*}
by sending $w \in \mathbb{B}_d$ to the character of point evaluation at $w$.
In the reverse direction, we obtain a continuous surjection
\begin{equation*}
  \pi: \mathcal{M}(\Mult(H^2_d)) \to \overline{\mathbb{B}_d}, \quad \chi \mapsto (\chi(z_1),\ldots,\chi(z_d)).
\end{equation*}
(The fact that $\pi$ takes values in $\overline{\mathbb{B}_d}$ follows from the fact
that the coordinate functions form a row contraction and that characters are completely contractive, see \cite[Proposition 3.8]{Paulsen02}, or from computation of the spectrum of the tuple of coordinate functions.)

The following result was shown by  Gleason, Richter and Sundberg, see \cite[Corollary 4.2]{GRS05}.
It shows that Gleason's problem can be solved in the multiplier algebra of $H^2_d$.
For background on Gleason's problem, see \cite[Section 6.6]{Rudin08}.

\begin{theorem}
  \label{thm:Gleason}
  Let $w \in \mathbb{B}_d$ and let $\varphi \in \Mult(H^2_d)$.
  Then there exist $\varphi_1,\ldots,\varphi_d \in \Mult(H^2_d)$ with
  \begin{equation*}
    \varphi(z) - \varphi(w) = \sum_{i=1}^d (z_i - w_i) \varphi_i(z) \quad \text{ for all } z \in \mathbb{B}_d.
  \end{equation*}
  Therefore, $\pi^{-1}(w)$ is the singleton consisting of the character of point evaluation at $w$.
\end{theorem}

The last sentence of Theorem \ref{thm:Gleason} is false for $d=\infty$; see \cite{DHS15a}.

Already in the case $d=1$, the fibers of $\pi$ over points in the boundary are known to be extremely complicated;
see \cite[Chapter 10]{Hoffman62} and \cite[Chapter X]{Garnett07}.
For general $d \ge 1$, the theory of interpolating sequences (specifically Theorem \ref{thm:is_subsequence} below)
implies that $\pi^{-1}(w)$ contains a copy of the Stone-\v{C}hech remainder $\beta \mathbb{N} \setminus \mathbb{N}$
for each $w \in \partial \mathbb{B}_d$; cf.\ the argument on p.\ 184 of \cite{Garnett07}. In particular, $\mathcal{M}(\Mult(H^2_d))$ is not metrizable
and has cardinality $2^{2^{\aleph_0}}$.

In summary, $\mathcal{M}(\Mult(H^2_d))$ contains a well understood part,
namely the open ball $\mathbb{B}_d$,
and a much more complicated part, namely $\pi^{-1}(\partial \mathbb{B}_d)$, i.e.\ characters in fibers over the boundary.
If $d=1$, then Carleson's corona theorem (\cite{Carleson62}, see also \cite[Chapter VIII]{Garnett07}) shows that the unit disc (i.e.\ the well understood part)
is dense in the maximal ideal space of $H^\infty = \Mult(H^2_1)$.
A deep result of Costea, Sawyer and Wick \cite{CSW11} shows that this remains true for higher $d$.

\begin{theorem}[Corona theorem for $H^2_d$]
  $\mathbb{B}_d$ is dense in $\mathcal{M}(\Mult(H^2_d))$.
\end{theorem}

In their proof, Costea, Sawyer and Wick use the complete Pick property of $H^2_d$
to reduce the original problem involving $\Mult(H^2_d)$ to a problem about $H^2_d$.
The point is that problems in $H^2_d$ are likely more tractable because one can exploit the
Hilbert space structure. Costea, Sawyer and Wick then manage to solve the (still extremely difficult)
Hilbert space problem using the function theory description of $H^2_d$, thus proving their corona theorem.
This result and its proof showcase again how different perspectives on the Drury--Arveson space can be helpful.

The translation of the corona problem into a Hilbert space problem, as well
as the proof of Theorem \ref{thm:Gleason}, rely on a factorization result for multipliers complete Pick spaces,
which has proved to be extremely useful for a variety of problems.
In the case of $H^2$, it was shown by Leech \cite{Leech14} (see also \cite{KR14} for a historical discussion).
References for the Drury--Arveson space and general complete Pick spaces
are \cite[Theorem 1.6]{EP02}, \cite[Theorem 8.57]{AM02} and \cite{BTV01}.
We write $M_{n,m}(\Mult(\mathcal{H}))$ for the space of all $n \times m$ matrices
with entries in $\Mult(\mathcal{H})$, which we regard as operators from $\mathcal{H}^m$ into $\mathcal{H}^n$.
Extending an earlier definition, we say that a function $L: X \times X \to M_{r}(\mathbb{C})$ is positive
if the block matrix $[L(x_i,x_j)]_{i,j=1}^n$ is positive semi-definite for any $x_1,\ldots,x_n \in X$.

\begin{theorem}
  \label{thm:Leech}
  Let $\mathcal{H}$ be a normalized complete Pick space with kernel $k$ and let
  $\Phi \in M_{r,n}(\Mult(\mathcal{H}))$ and $\Theta \in M_{r,m}(\Mult(\mathcal{H}))$ be given.
  The following assertions are equivalent:
  \begin{enumerate}[label=\normalfont{(\roman*)}]
    \item there exists $\Psi \in M_{n,m}(\Mult(\mathcal{H}))$ of norm at most one such that
     $\Phi \Psi = \Theta$;
   \item the function $(z,w) \mapsto k(z,w) (\Phi(z) \Phi(w)^* - \Theta(z) \Theta(w)^*)$ is positive;
   \item the operator inequality $M_\Phi M_\Phi^* \ge M_\Theta M_\Theta^*$ holds.
  \end{enumerate}
\end{theorem}

\begin{proof}[Sketch of proof]
  The equivalence of (ii) and (iii) follows by testing the inequality on finite linear combinations
  of vectors in $\mathcal{H}^r$ whose entries are reproducing kernels.

  If (i) holds, then
  \begin{equation*}
    M_{\Theta} M_{\Theta}^* = M_{\Phi} M_{\Psi} M_{\Psi}^* M_{\Phi}^* \le M_{\Phi} M_{\Phi}^*,
  \end{equation*}
  so (iii) holds.
  
  The main work occurs in the proof of (ii) $\Rightarrow$ (i); it is here where the complete Pick property is used.
  This implication can be shown by using a lurking a lurking isometry argument,
  similar to the proof of Theorem \ref{thm:BTV};
  see \cite[Theorem 8.57]{AM02} for the details.
\end{proof}

\begin{remark}
  \begin{enumerate}[label=\normalfont{(\alph*)},wide]
    \item Taking $n=r=m=1$ and $\Phi=1$ in Theorem \ref{thm:Leech}, we essentially recover the basic multiplier criterion (Proposition \ref{prop:mult_crit}).
    \item In the setting of (iii) of Theorem \ref{thm:Leech}, the Douglas factorization lemma \cite{Douglas66}
      always yields a contraction $T: \mathcal{H}^m \to \mathcal{H}^n$ so that $M_{\Phi} T = M_{\Theta}$;
      this has nothing to do with multiplication operators or the complete Pick property.
      The crucial point in Theorem \ref{thm:Leech} is that $T$ can be chosen to be a multiplication operator again.
      Thus, Theorem \ref{thm:Leech} can be regarded as a version of the Douglas factorization lemma
      for the multiplier algebra of a complete Pick space.
      Under suitable assumptions, the complete Pick property is necessary for the validity
      of such a Douglas factorization lemma; see \cite{MT12} for a precise statement.
    \item There is a version of Theorem \ref{thm:Leech} for infinite matrices; see \cite[Theorem 8.57]{AM02} for the precise statement.
  \end{enumerate}
\end{remark}

Let us see how Leech's Theorem can be used to show Theorem \ref{thm:Gleason}.

\begin{proof}[Sketch of proof of Theorem \ref{thm:Gleason}]
  Let $w = 0$. We have to show that if $\varphi \in \Mult(H^2_d)$ with $\varphi(0) = 0$,
  then there exist $\varphi_1,\ldots,\varphi_d \in \Mult(H^2_d)$ with $\varphi = \sum_{i=1}^d z_i \varphi_i$.
  We may assume that $\|\varphi\|_{\Mult(H^2_d)} \le 1$.
  Since $\varphi(0) = 0$, the function $\varphi$ is a contractive multiplier
  from $H^2_d$ into $\{f \in H^2_d: f(0) = 0 \}$.
  The reproducing kernel for this last space is $K-1$, where $K$ denotes the reproducing kernel of $H^2_d$.
  A small modification of the basic multiplier criterion (Proposition \ref{prop:mult_crit},
  see for instance \cite[Theorem 5.21]{PR16} for the general statement) then shows that
  \begin{equation*}
    (K(z,w) - 1) - K(z,w) \varphi(z) \overline{\varphi(w)} =
    \frac{1}{1 - \langle z,w \rangle } ( \langle z,w \rangle  - \varphi(z) \varphi(w))
  \end{equation*}
  is positive as a function of $(z,w)$.
  In this setting, we may apply (ii) $\Rightarrow$ (i) of Theorem \ref{thm:Gleason} with
  $n=d$, $r=m=1$,
  $\Phi(z) =
  \begin{bmatrix}
    z_1 & \cdots & z_d
  \end{bmatrix}$ and $\Theta = \varphi$ to find a multiplier $\Psi \in M_{d,1}(\Mult(H^2_d))$
  with $\varphi = \Phi \Psi$. Writing $\Psi$ in terms of its components, the result in the case $w=0$ follows.

  The general case of $w \in \mathbb{B}_d$ can be deduced from the case $w=0$ by using biholomorphic automorphisms
  of the ball; see \cite[Corollary 4.2]{GRS05}. Alternatively, the proof above can be extended
  to arbitrary $w \in \mathbb{B}_d$ with the help of some multivariable spectral theory;
  see \cite[Proposition 8.5]{Hartz17a}.

  The additional statement follows from the first part. Indeed, let $\chi \in \mathcal{M}(\Mult(H^2_d))$
  with $w = \pi(\chi) \in \mathbb{B}_d$. Given $\varphi \in \Mult(H^2_d)$, write
  \begin{equation*}
    \varphi(z) - \varphi(w) = \sum_{i=1}^d (z_i - w_i) \varphi_i(z)
  \end{equation*}
  as in the first part and apply $\chi$ to both sides. Since $\chi(z_i) = w_i$ for $1 \le i \le d$
  and $\chi(1) = 1$, it follows that $\chi(\varphi) = \varphi(w)$.
  Thus, $\chi$ is the character of evaluation at $w$.
\end{proof}

\begin{remark}
  Theorem \ref{thm:Gleason} can also be proved by using the non-commutative approach to multipliers explained
  in Subsection \ref{ss:nc_mult}; see \cite[Proposition 3.2]{DRS15} and \cite[Corollary 2.6]{DP98a}.
\end{remark}

Next, we discuss how Leech's Theorem can be used in the translation of the corona problem
into a Hilbert space problem.
First, it is a routine exercise in Gelfand theory to translate the corona problem into a more concrete
question about multipliers that are jointly bounded below.

\begin{proposition}
  \label{prop:corona_gelfand}
    The following are equivalent:
    \begin{enumerate}[label=\normalfont{(\roman*)}]
      \item $\mathbb{B}_d$ is dense in $\mathcal{M}(\Mult(H^2_d))$;
      \item whenever $\varphi_1,\ldots,\varphi_n \in \Mult(H^2_d)$ with $\sum_{i=1}^n |\varphi_i|^2 \ge \varepsilon^2 > 0$, there exist $\psi_1,\ldots,\psi_n \in \Mult(H^2_d)$ with $\sum_{i=1}^n \varphi_i \psi_i = 1$.
    \end{enumerate}
\end{proposition}

For a proof, see \cite[Lemma B.9.2.6]{Nikolski02a} or \cite[Theorem V.1.8]{Garnett07}.

Leech's theorem makes it possible to translate condition (ii),
which involves multipliers, into a condition involving only functions in $H^2_d$.
This result is often called the Toeplitz corona theorem.
In the case of $H^2$, it was proved by Arveson \cite{Arveson75}.
In the case of the Drury--Arveson space, the result appears in \cite[Section 5]{BTV01} and in \cite[Theorem 2.2]{EP02}.

\begin{theorem}[Toeplitz corona theorem]
  \label{thm:Toeplitz_corona}
  Let $\mathcal{H}$ be a normalized complete Pick space with kernel $k$.
  Let $\varphi_1,\ldots,\varphi_n \in \Mult(\mathcal{H})$. The following are equivalent:
  \begin{enumerate}[label=\normalfont{(\roman*)}]
    \item there exist $\psi_1,\ldots,\psi_n \in \Mult(\mathcal{H})$ with $\sum_{i=1}^n \varphi_i \psi_i = 1$;
    \item there exists $\delta > 0$ such that
      \begin{equation*}
      (z,w) \mapsto k(z,w) \Big( \sum_{i=1}^n \varphi_i(z) \overline{\varphi_i(w)} - \delta^2 \Big)
      \end{equation*}
      is positive.
    \item the row operator $
      \begin{bmatrix}
        M_{\varphi_1} & \cdots & M_{\varphi_n}
      \end{bmatrix}: \mathcal{H}^n \to \mathcal{H}$ is surjective.
  \end{enumerate}
\end{theorem}

\begin{proof}
  We may normalize so that the row $\Phi :=
  \begin{bmatrix}
    \varphi_1 & \ldots & \varphi_n
  \end{bmatrix}$ has multiplier norm at most one.
  Then Theorem \ref{thm:Leech}, applied with with $r=m=1$ and $\Theta = \delta$, shows
  that there exists a contractive column multiplier $\Psi \in M_{n,1}(\Mult(\mathcal{H}))$
  so that $\Phi \Psi = \delta$ if and only if the mapping in (ii) is positive.
  This shows the equivalence of (i) and (ii).
  
  The equivalence with (iii) then follows from Theorem \ref{thm:Leech} and the basic operator theory fact that
  $M_\Phi$ is surjective if and only if $M_{\Phi} M_{\Phi}^* \ge \delta^2 I$ for some $\delta > 0$,
  which is a consequence of the open mapping theorem; see \cite{Douglas66} or \cite[Theorem 4.13]{Rudin91}
  for a more general Banach space result.
  (Note that the implication (i) $\Rightarrow$ (iii) also follows directly from the observation
  that the column of the $\psi_i$ is a right inverse of the row of the $\varphi_i$.)
\end{proof}

\begin{remark}
  \label{rem:corona_norm}
    The proof of Theorem \ref{thm:Toeplitz_corona} gives quantitative information
      on the norm of the $\psi_i$.
      If we normalize so the row of the $\varphi_i$ has norm one and $\delta$ is as in (ii),
      then we may achieve that the norm of the column of the $\psi_i$ is at most $\frac{1}{\delta}$,
      and this is best possible.
\end{remark}

In order to prove the corona theorem for $H^2_d$, one therefore has to show that
if $\varphi_1,\ldots,\varphi_n \in \Mult(H^2_d)$ with 
\begin{equation}
  \label{eqn:cor_1}
  \sum_{i=1}^n |\varphi_i|^2 \ge \varepsilon^2 > 0,
\end{equation}
then there exists $\delta > 0$ such that
\begin{equation}
  \label{eqn:cor_2}
  (z,w) \mapsto \frac{\sum_{i=1}^n \varphi_i(z) \overline{\varphi_i(w)} - \delta^2}{1 - \langle z,w \rangle }
\end{equation}
is positive. This was achieved by Costea, Sawyer and Wick \cite{CSW11}.
Note that testing \eqref{eqn:cor_2} on the diagonal $\{z=w\}$, one recovers \eqref{eqn:cor_1} with $\varepsilon = \delta$.
Letting $\Phi =
\begin{bmatrix}
  \varphi_1 & \cdots & \varphi_n
\end{bmatrix}$,
the passage from \eqref{eqn:cor_1} to \eqref{eqn:cor_2} is equivalent to proving that
\begin{equation*}
  \langle M_\Phi M_\Phi^* K_z,K_z \rangle \ge \varepsilon^2 \|K_z\|^2 \quad \text{ for all } z \in \mathbb{B}_d
\end{equation*}
implies that there exists $\delta > 0$ such that
\begin{equation*}
  \langle M_\Phi M_\Phi^* f,f \rangle \ge \delta^2 \|f\|^2 \quad \text{ for all } f \in H^2_d.
\end{equation*}
The philosophy that in order to show certain properties of operators, it should suffice
to test the property on kernel functions, is sometimes called
the \emph{reproducing kernel thesis}; see for instance \cite[Chapter 4]{ARS+19} and \cite[Section A.5.8]{Nikolski02a}.

\begin{remark}
  It is natural to ask for the dependence of $\delta$ in \eqref{eqn:cor_2} on $\varepsilon$ in \eqref{eqn:cor_1}.
  By Remark \ref{rem:corona_norm}, this leads to bounds on the norm of the corona solutions
  in terms of the corona data. For a discussion of bounds in the classical case $d=1$, see \cite[Appendix 3]{Nikolskiui86} and \cite{Treil02}.
\end{remark}

Given the Toeplitz corona theorem for complete Pick spaces and the Costea--Sawyer--Wick corona theorem for $H^2_d$, one may wonder if the
corona theorem holds for all complete Pick spaces. Without further assumptions, this is not the case,
which was shown by Aleman, \mcc, Richter and the author; see \cite[Theorem 5.5]{AHM+17a}.

  \begin{theorem}
    There exists a normalized complete Pick space $\mathcal{H}$ with the following properties:
    \begin{enumerate}[label=\normalfont{(\alph*)}]
      \item $\mathcal{H}$ consists of holomorphic functions on $\mathbb{D}$,
      \item $\mathcal{H}$ contains all functions holomorphic in a neighborhood of $\overline{\mathbb{D}}$, and
      \item $\mathbb{D}$ is \emph{not dense} in
    $\mathcal{M}(\Mult(\mathcal{H}))$.
    \end{enumerate}
  \end{theorem}

  The proof uses a construction of Salas \cite{Salas81}, so the space $\mathcal{H}$ was called the \emph{Salas space} 
  in \cite{AHM+17a}.
  The Salas space is in fact contained in the disc algebra.
  Observe that condition (b) rules out trivial constructions such as taking the restriction 
  of the Hardy space on $2 \mathbb{D}$ to $\mathbb{D}$.
  (In \cite[Theorem 5.5]{AHM+17a}, a condition that is slightly different from (b) is stated;
  but (b) follows from the fact that the reproducing kernel of $\mathcal{H}$ is given
  by a power series with coefficients $a_n$ satisfying $\frac{a_n}{a_{n+1}} \to 1$.)

 \subsection{Interpolating sequences}
 \label{ss:is}

 Next, we will discuss interpolating sequences.
 The classical definition in the disc is the following.

    \begin{definition}
      A sequence $(z_n)$ in $\mathbb{D}$ is an \emph{interpolating sequence (IS)} for $H^\infty$ if
      the map
      $H^\infty \to \ell^\infty, \varphi \mapsto (\varphi(z_n))$,
    is surjective.
    \end{definition}
    The study of interpolating sequences sheds considerable light on the structure of the maximal ideal space of $H^\infty$; see for instance \cite[Chapter X]{Garnett07}.
    They also made an appearance in the description of Douglas algebras,
    which are subalgebras of $L^\infty$ containing $H^\infty$.
    The description is due to Chang \cite{Chang76} and Marshall \cite{Marshall76a}; see also \cite[Chapter IX]{Garnett07}. For general background on interpolating sequences, see \cite[Chapter VII]{Garnett07},
    \cite{Seip04} and \cite[Chapter 9]{AM02}.

    At first, it is not even obvious that interpolating sequences in $H^\infty$ exist,
    but they do, and in fact they were characterized by Carleson \cite{Carleson58}.
    His characterization involves two more conditions:

    \begin{itemize}
      \item[(WS)] $(z_n)$ is \emph{weakly separated} if there exists $\varepsilon > 0$ so that
        \begin{equation*}
          \Big| \frac{z_n - z_m}{ 1 - \overline{z_m} z_n} \Big| \ge \varepsilon \quad \text{ whenever } n \neq m.
        \end{equation*}
\item[(C)] $(z_n)$ satisfies the \emph{Carleson measure condition} if there exists $C \ge 0$ so that
  \begin{equation*}
    \sum_{n} (1 - |z_n|^2) |f(z_n)|^2 \le C \|f\|^2_{H^2} \quad \text{ for all } f \in H^2.
  \end{equation*}
    \end{itemize}

    Note that the quantity appearing in condition (WS) is the pseudohyperbolic metric on $\mathbb{D}$,
    which we already encountered in Example \ref{exa:Pick_small}.
    The Carleson measure condition means that the measure $\sum_n (1 - |z_n|^2) \delta_{z_n}$
    is a Carleson measure for $H^2$.
    \begin{theorem}[Carleson's interpolation theorem]
      In $H^\infty$, (IS) $\Leftrightarrow$ (WS) + (C)
    \end{theorem}

    The definition of an interpolating sequence, as well as the conditions
    involved in Carleson's theorem, naturally carry over to the Drury--Arveson space,
    and in fact to any complete Pick space.

\begin{definition}
Let $\mathcal{H}$ be a complete Pick space on $X$. A sequence $(z_n)$ in $X$ is said to
\begin{itemize}
\item[(IS)] be
  an \emph{interpolating sequence} (for $\Mult(\mathcal{H})$) if the map
\begin{equation*}
  \Mult(\mathcal{H}) \to \ell^\infty, \varphi \mapsto (\varphi(z_n)),
\end{equation*}
is surjective;
\item[(WS)] be \emph{weakly separated} if there exists $\varepsilon > 0$ so that 
  \begin{equation*}
    1 - \frac{ |k(z_n,z_m)|^2}{k(z_n,z_n) k(z_m,z_m)} \ge \varepsilon \quad \text{ whenever } n \neq m.
  \end{equation*}
\item[(C)]
satisfy the \emph{Carleson measure condition} if there exists $C \ge 0$ so that
\begin{equation*}
\sum_{n} \frac{|f(z_n)|^2}{k(z_n,z_n)} \le C \|f\|^2_{\mathcal{H}}
\quad \text{ for all } f \in \mathcal{H}.
\end{equation*}
\end{itemize}
\end{definition}

The analogue of Carleson's theorem in the Dirichlet space
was established independently by Bishop \cite{Bishop94} and by Marshall and Sundberg \cite{MS94a}.
The proof of Marshall and Sundberg crucially used the complete Pick property of the Dirichlet space.
They established the following theorem, see \cite[Corollary 7]{MS94a} and \cite[Theorem 9.19]{AM02}.
Recall that a sequence $(x_n)$ of vectors in a Hilbert space is said to be a \emph{Riesz sequence}
if there exist constants $C_1,C_2 > 0$ with
\begin{equation*}
  C_1 \Big( \sum_n |\alpha_n|^2 \Big) \le \Big\| \sum_n \alpha_n x_n \Big\|^2 \le C_2 \Big( \sum_n |\alpha_n|^2 \Big)
\end{equation*}
for all sequences $(\alpha_n)$ of scalars.
If the linear span of the $x_n$ is dense, the Riesz sequence is called a \emph{Riesz basis}.

\begin{theorem}
  \label{thm:MS_int}
  Let $\mathcal{H}$ be a normalized complete Pick space on $X$ with kernel $k$. Then a sequence $(z_n)$ in $X$
  is an interpolating sequence if and only if the normalized reproducing kernels $\{ k_{z_n} / \|k_{z_n}\|: n \in \mathbb{N}\}$ form a Riesz sequence in $\mathcal{H}$.
\end{theorem}

\begin{proof}
  We prove sufficiency.
  Suppose that the normalized reproducing kernels form a Riesz sequence in $\mathcal{H}$.
  Let $Z = \{z_n: n \in \mathbb{N}\}$ and consider the space $\mathcal{H} \big|_Z$.
  Then the restrictions to $Z$ of the normalized reproducing kernels form a Riesz basis
  of $\mathcal{H} \big|_Z$. Thus, for every $(w_n) \in \ell^\infty$, we obtain a bounded linear map
  \begin{equation*}
    T_w: \mathcal{H} \big|_Z \to \mathcal{H} \big|_Z, \quad k_{z_n} \big|_Z \mapsto \overline{w_n} k_{z_n} \big|_Z.
  \end{equation*}
  The adjoint $T_w^*$ is multiplication operator on $\mathcal{H} \big|_Z$ with symbol
  $\varphi$ given by $\varphi(z_n) = w_n$.
  This shows that
  \begin{equation*}
    \Mult(\mathcal{H} \big|_Z) \to \ell^\infty, \quad \varphi \mapsto (\varphi(z_n)),
  \end{equation*}
  is surjective. By the complete Pick property of $X$, the restriction map $\Mult(\mathcal{H}) \to \Mult(\mathcal{H} \big|_Z)$ is surjective, so $Z$ is an interpolating sequence for $\Mult(\mathcal{H})$.

  Necessity is true independently of the complete Pick property; see for instance \cite[Theorem 9.19]{AM02} for a proof.
\end{proof}

Theorem \ref{thm:MS_int} is another instance where the complete Pick property of a space $\mathcal{H}$ can be used to translate
a problem about $\Mult(\mathcal{H})$ into a more tractable Hilbert space problem in $\mathcal{H}$.

Theorem \ref{thm:MS_int} has the following consequence regarding the existence of interpolating sequences.

\begin{theorem}
  \label{thm:is_subsequence}
  Let $\mathcal{H}$ be a normalized complete Pick space on $X$ with kernel $k$.
  If $(z_n)$ is a sequence in $X$ such that $\lim_{n \to \infty} k(z_n,z_n) = \infty$,
  then $(z_n)$ has subsequence that is interpolating.
\end{theorem}

\begin{proof}
  The assumption implies that $\lim_{n \to \infty} \|k_{z_n}\| = \infty$.
  Using the particular form of the kernel $k(z,w) = \frac{1}{1 - \langle b(z), b(w) \rangle }$
  given by Theorem \ref{thm:AM}, we see that $\lim_{n \to \infty} \frac{k_{z_n}(w)}{\|k_{z_n}\|} = 0$
  for all $w \in X$; hence $(k_{z_n} / \|k_{z_n}\|)$ is a sequence of unit vectors in $\mathcal{H}$ that
  tends to zero weakly.
  From this, one can recursively extract a subsequence of $(k_{z_n} / \|k_{z_n}\|)$ that is a Riesz sequence
  (see for example \cite[Proposition 2.1.3]{AK06} for a more general statement), so the result follows from Theorem \ref{thm:MS_int}.
\end{proof}

The proof above is taken from \cite[Proposition 5.1]{AHM+17}. For a direct proof using Pick matrices,
see \cite[Proposition 9.1]{DHS15} and also \cite[Lemma 4.6]{Serra04}.

Theorem \ref{thm:is_subsequence} in particular yields the existence of interpolating sequences
in the Drury--Arveson space and in the Dirichlet space.
The existence of interpolating sequences in the Dirichlet space was proved earlier by Axler \cite{Axler92}.

After the theorem of Bishop and of Marshall and Sundberg, the scope of the characterization of interpolating sequences was further extended by B\o e \cite{Boe05} (see also \cite{Boee02}),
but this still left open the case of the Drury--Arveson space.
Using work of Agler and \mcc\ \cite[Chapter 9]{AM02},
Carleson's characterization was extended to all complete Pick spaces, including the Drury--Arveson space,
by Aleman, Richter, \mcc\ and the author \cite{AHM+17}.

\begin{theorem}
  \label{thm:IS}
  In any normalized complete Pick space, (IS) $\Leftrightarrow$ (WS) + (C).
\end{theorem}

The first proof of this result used the solution of the Kadison--Singer problem due to Marcus, Spielman
and Srivastava \cite{MSS15}. Explicitly, it used an equivalent formulation of the Kadison--Singer problem,
known as the Feichtinger conjecture. In the context of interpolating sequences,
the Feichtinger conjecture, combined with Theorem \ref{thm:MS_int}, implies that every sequence
satisfying the Carleson measure condition is a finite union of interpolating sequences.

In the case of $H^2_d$ for finite $d$, a proof avoiding the use of the Marcus--Spielman--Srivastava
theorem was given in \cite{AHM+17,AHM+18}.
The key property making this possible is the so-called column-row property of $H^2_d$.
To explain this property, recall that in $H^2_d$, we have
\begin{equation*}
      \big\|\begin{bmatrix}
      M_{z_1} & \cdots & M_{z_d}
    \end{bmatrix} \big\| = 1.
\end{equation*}
On the other hand, one may compute that
\begin{equation*}
    \Bigg\|
    \begin{bmatrix}
      M_{z_1} \\ \vdots \\ M_{z_d}
    \end{bmatrix} \Bigg\| = \sqrt{d}.
\end{equation*}
In particular, the column norm of a tuple of multiplication operators may exceed the row norm.
In fact, one can iterate this example to construct a sequence of multipliers that are bounded as a row,
but unbounded as a column; see \cite[Subsection 4.2]{AHM+18} or \cite[Lemma 4.8]{CH19}.
What about the converse?

\begin{definition}
  An RKHS $\mathcal{H}$ satisfies the column-row property with constant $c \ge 1$ if
  \begin{equation*}
    \big\|
    \begin{bmatrix}
      M_{\varphi_1} & M_{\varphi_2} & \cdots
    \end{bmatrix} \big\| \le c
    \Bigg\|
    \begin{bmatrix}
    M_{\varphi_1} \\ M_{\varphi_2} \\ \vdots
    \end{bmatrix} \Bigg\|
  \end{equation*}
  for all sequences $(\varphi_n)$ in $\Mult(\mathcal{H})$.
\end{definition}

It was shown in \cite[Remark 3.7]{AHM+17} that if $\mathcal{H}$ is a complete Pick space satisfying the column-row
property, then one obtains a direct proof of Theorem \ref{thm:IS} that does not rely on the Marcus--Spielman--Srivastava theorem. The column-row property also came up in other contexts, such as in the theory
of weak product spaces, see Subsection \ref{ss:wp} below.
The first non-trivial example of a complete Pick space satisfying the column-row property was the Dirichlet space,
a result due to Trent \cite{Trent04}, see also \cite{KT13}. Generalizing Trent's argument,
it was shown in \cite{AHM+18}, see also \cite{AHM+18a}, that if $d < \infty$, then $H^2_d$ satisfies the column-row property with some constant $c_d$. Thus, one obtains a direct proof of Theorem \ref{thm:IS} for $H^2_d$.
Finally, the column-row property was shown to be automatic for all complete Pick spaces in \cite{Hartz20}.

\begin{theorem}
  Every normalized complete Pick space satisfies the column-row property with constant $1$.
\end{theorem}

 \subsection{Weak products and Hankel operators}
 \label{ss:wp}

The Hardy space $H^2$ is typically not studied as an isolated object, but
as a member of the scale of $H^p$ spaces.
Of particular importance are $H^1$ and $H^\infty$.
The analogue of $H^\infty$ in the context of the Drury--Arveson space
is generally considered to be the multiplier algebra $\Mult(H^2_d)$.
What is an appropriate analogue of $H^1$?

The classical function theory definition of $H^1$ is
\begin{equation*}
  H^1 = \Big\{ f \in \mathcal{O}(\mathbb{D}): \sup_{0 \le r < 1} \int_{0}^{2 \pi} |f (r e^{i t})| \, dt < \infty \Big\}.
\end{equation*}
However, this definition does not generalize well to the case of the Drury--Arveson space.
Instead, we may take inspiration from the classical fact that
\begin{equation*}
  H^1 = \{f \cdot g: f,g \in H^2\}.
\end{equation*}
When trying to define an $H^1$ space for more general RKHS in a similar way,
one encounters the following obstacle: It is not obvious that the collection
of products of two functions in an RKHS forms a vector space, let alone a Banach space.

Coifman, Rochberg and Weiss \cite{CRW76} circumvented this issue by allowing infinite sums of products.

  \begin{definition}
    Let $\mathcal{H}$ be an RKHS. The \emph{weak product space} of $\mathcal{H}$ is
  \begin{equation*}
    \mathcal{H} \odot \mathcal{H} = \Big\{ h = \sum_{n=1}^\infty f_n g_n: \sum_{n = 1}^\infty \|f_n\| \|g_n\| < \infty \Big\}.
  \end{equation*}
\end{definition}
Equipped with the norm
\begin{equation*}
  \|h\|_{\mathcal{H} \odot \mathcal{H}} = \inf \Big\{ \sum_{n=1}^\infty \|f_n\| \|g_n\| : h = \sum_{n=1}^\infty f_n g_n \Big\},
\end{equation*}
the weak product space $\mathcal{H} \odot \mathcal{H}$ is a Banach function space.
We think of $\mathcal{H} \odot \mathcal{H}$ as playing the role of $H^1$.

Classically, $H^1$ plays a key role in the study of Hankel operators on $H^2$.
Hankel operators (more precisely, little Hankel operators)
can also be defined on the Drury--Arveson space, and more generally on any normalized complete Pick space.

    \begin{definition}
      Let $\mathcal{H}$ be an RKHS that densely contains $\Mult(\mathcal{H})$.
      A Hankel form on $\mathcal{H}$ with symbol $b \in \mathcal{H}$ is a densely defined bilinear form
      \begin{equation*}
        B_b: \mathcal{H} \times \mathcal{H} \to \mathbb{C}, \quad (\varphi,f) \mapsto \langle \varphi f, b \rangle
        \quad (\varphi \in \Mult(\mathcal{H}), f \in \mathcal{H}).
      \end{equation*}
      Let $\Han(\mathcal{H}) = \{b \in \mathcal{H}: B_b \text{ is bounded}\}$.
    \end{definition}

    This definition makes sense in particular in the Drury--Arveson space, and more generally
    in any normalized complete Pick space, since the kernel functions in a normalized complete
    Pick space are multipliers.

    If $b \in \Han(\mathcal{H})$, then we obtain a bounded operator $H_b: \mathcal{H} \to \overline{\mathcal{H}}$ with
    \begin{equation*}
      \langle f, \overline{H_b \varphi} \rangle = \langle \varphi f, b \rangle.
    \end{equation*}
    Here, $\overline{\mathcal{H}}$ denotes the conjugate Hilbert space of $\mathcal{H}$.
    It follows that
    \begin{equation*}
      H_b M_\psi = M_{\overline{\psi}}^* H_b
    \end{equation*}
    for all $\psi \in \Mult(\mathcal{H})$, so we think of $H_b$ as a (little) Hankel operator.

    Nehari's theorem \cite{Nehari57} shows that the dual space of $H^1$ can be identified with the space $\Han(H^2)$
    of symbols of bounded Hankel operators.
    The dual of $H^1$, was then characterized in function theoretic
    terms by Fefferman as $\textup{BMOA}$, the space of functions in $H^2$ with bounded mean oscillation;
    see \cite{Fefferman71,FS72} and also \cite[Chapter VI]{Garnett07}.
    Thus, we think of $\Han(\mathcal{H})$ as an analogue of $\textup{BMOA}$.

    Nehari's theorem was extended to the Drury--Arveson space by Richter and Sundberg \cite{RS14},
    and then to all normalized complete Pick spaces in \cite{AHM+18,Hartz20}.

    \begin{theorem}
      \label{thm:Han_dual}
      Let $\mathcal{H}$ be a normalized complete Pick space. Then 
      \begin{equation*}
        (\mathcal{H} \odot \mathcal{H})^* \cong \Han(\mathcal{H}).
      \end{equation*}
    \end{theorem}

    At this stage, we have reasonable analogues of $H^1,H^2,H^\infty$ and $\textup{BMOA}$
    for the Drury--Arveson space, and in fact for any normalized complete Pick space $\mathcal{H}$,
    namely $\mathcal{H} \odot \mathcal{H}, \mathcal{H}, \Mult(\mathcal{H})$ and $\Han(\mathcal{H})$,
    respectively.
    One can then use complex interpolation between $\mathcal{H} \odot \mathcal{H}$
    and $\Han(\mathcal{H})$ to define an $\mathcal{H}^p$ scale, see \cite{AHM+20}.
    Theorem \ref{thm:Han_dual} can be used to show that one recovers the original Hilbert space $\mathcal{H}$
    in the middle. Moreover, the functions in $\mathcal{H}^p$ have the correct order of growth,
    but many questions remain open; see \cite{AHM+20}.

    Returning to the definition of the weak product space, it is natural to ask if the sum is actually
    required. Remarkably, the answer turns out to be no.
    This was proved by Jury and Martin \cite{JM18} for all complete Pick spaces satisfying
    the column-row property, which we now know to be all complete Pick spaces \cite{Hartz20}.

    \begin{theorem}
      \label{thm:JM_WP}
      Let $\mathcal{H}$ be a complete Pick space.
      Then each $h \in \mathcal{H} \odot \mathcal{H}$ factors as
      $h = f g$ for some $f,g \in \mathcal{H}$ with
      $\|h\|_{\mathcal{H} \odot \mathcal{H}} = \|f\| \|g\|$.
  \end{theorem}

  The proof of Jury and Martin uses the non-commutative approach
  to the Drury--Arveson space, and no commutative proof is currently known.
  This becomes even more remarkable when considering that before the work of Jury and Martin,
  Theorem \ref{thm:JM_WP} was not even known for the Dirichlet space, a space of one-variable
  holomorphic functions.

  It is a basic fact that the multiplier algebras of $H^2$ and of $H^1$ are both $H^\infty$,
  and in particular agree.
  As was discussed in Subsection \ref{ss:ft_mult}, the description of multipliers in the Drury--Arveson space is more involved,
  but one may still ask if $\Mult(\mathcal{H} \odot \mathcal{H}) = \Mult(\mathcal{H})$.
  A positive answer was established by Richter and Wick \cite{RW16} in the case of the Dirichlet space
  and of the Drury--Arveson space $H^2_d$ for $d \le 3$.
  Using operator space techniques and dilation theory, Clou\^atre and the author \cite{CH19}
  gave a positive answer for all complete Pick spaces satisfying the column-row property,
  which are now known to be all complete Pick spaces.

    \begin{theorem}
      Let $\mathcal{H}$ be a complete Pick space. Then
      $\Mult(\mathcal{H} \odot \mathcal{H}) = \Mult(\mathcal{H})$.
    \end{theorem}

\bibliographystyle{amsplain}
\bibliography{da_invitation_literature}

\end{document}